\documentclass[opre,sglanonrev]{informs4}

\RequirePackage{tgtermes}
\RequirePackage{newtxtext}
\RequirePackage{newtxmath}
\RequirePackage{bm}
\RequirePackage{endnotes}

\OneAndAHalfSpacedXII 

\usepackage{algorithm}
\usepackage{algpseudocode}
\usepackage{tikz}

\usepackage{natbib}
 \bibpunct[, ]{(}{)}{,}{a}{}{,}%
 \def\bibfont{\small}%
 \def\bibsep{\smallskipamount}%
 \def\bibhang{24pt}%
 \def\newblock{\ }%
 \def\BIBand{and}%

\EquationsNumberedThrough    

\TheoremsNumberedThrough     
\ECRepeatTheorems  %

\usepackage{fix-cm}
\usepackage{subcaption}
\usepackage{optidef}
\usepackage{longtable}
\usepackage{thm-restate}
\usepackage{comment}
\MANUSCRIPTNO{MOOR-0001-2024.00}

\newcommand{\esg}[1]{\textcolor{magenta}{#1}} 

\usepackage{tikz,multirow,multicol}
\usepackage{enumitem}
\usetikzlibrary{arrows.meta, positioning}
\usepackage{algorithm, algpseudocode}
\usepackage{soul}
\MANUSCRIPTNO{MOOR-0001-2024.00}

\begin{document}


\RUNAUTHOR{Sindermann, Gel, Erkip}

\RUNTITLE{Optimal Replenishment Policies for IVMs}

\TITLE{Optimal Replenishment Policies for Industrial Vending Machines}

\ARTICLEAUTHORS{%
\AUTHOR{Karina M. Sindermann, Esma S. Gel\thanks{Corresponding Author, esma.gel@unl.edu}}
\AFF{College of Business, University of Nebraska-Lincoln, \EMAIL{ksindermann2@unl.edu}, \EMAIL{esma.gel@unl.edu}}
\AUTHOR{Nesim Erkip}
\AFF{Department of Industrial Engineering, Bilkent University, Ankara, Turkey, \EMAIL{nesim@bilkent.edu.tr}, \URL{}}
}

\ABSTRACT{%
Industrial Vending Machines (IVMs) automate the dispensing of a variety of supplies like safety equipment and tools at customer sites, providing 24/7 access while tracking inventory in real-time. Industrial distribution companies typically manage the replenishment of IVMs using periodic schedules, which do not take advantage of these advanced real-time monitoring capabilities. We develop two approaches to optimize the long-term average cost of replenishments and stockouts per unit time: a state-dependent optimal control policy that jointly considers all inventory levels (referred to as {\it trigger set policy}) and a fixed cycle policy that optimizes replenishment frequency. We prove the monotonicity of the optimal trigger set policy and leverage it to design a computationally efficient approximate online control framework. Unlike existing methods, which typically handle a very limited number of items due to computational constraints, our approach scales to hundreds of items while achieving near-optimal performance. 
Leveraging transaction data from our industrial partner, we conduct an extensive set of numerical experiments to demonstrate this claim. Our results show that optimal fixed cycle replenishment reduces costs by 61.7 to 78.6\% compared to current practice, with our online control framework delivering an additional 4.1 to 22.9\% improvement. Our novel theoretical results provide practical tools for effective replenishment management in this modern vendor-managed inventory context. }

\HISTORY{\today}
\maketitle

\section{Introduction}
\label{sec:Introduction}
Industrial Vending Machines (IVMs) have emerged as an important tool in the industrial distribution industry, which focuses on providing companies with the industrial supplies, equipment, and consumables needed to keep their operations running smoothly. 
This industry is a significant sector of the economy. The broader industrial distribution market generated approximately \$2.5 trillion in revenue in 2017 \citep{AbdelnourShakeout2019}. Within the industrial distribution market, the more narrowly defined industrial supplies wholesale segment alone has been estimated to generate \$111.7 billion in revenue by 2024, having grown at a compound annual growth rate of 2.6\% over the past five years \citep{IBISWorld2024}. 

In the industrial distribution industry, the relationship between suppliers and their customers, typically manufacturing firms, often follows a Vendor Managed Inventory (VMI) model. In these arrangements, suppliers take responsibility for managing inventory replenishment at the customer’s location, determining both the timing and quantity of replenishments based on real-time inventory levels and demand data~\citep{darwish2010vendor}.  Despite rapid growth, the industry faces significant challenges, including digital disruption, rising customer expectations, and new market entry~\citep{AbdelnourShakeout2019}. In response, leading distributors are moving beyond traditional product distribution to offer value-added services and embrace digital transformation. 

The use of IVMs represents an innovative solution that addresses these industry trends. Unlike traditional vending machines, IVMs in B2B context are sophisticated systems designed for workplace environments, providing controlled access to industrial tools, equipment, consumables, safety equipment and other industrial supplies~\citep{fastenal2012business}. IVMs enhance VMI capabilities by offering real-time data collection, automated replenishment triggers, and item-level tracking, going beyond what is typically possible in traditional VMI systems. Real-time inventory tracking enables automated reordering and reduces stockout risks~\citep{falasca2018success,KROS2019100506}. These systems support data-driven decision making, facilitating trend identification and purchasing optimization. IVMs improve accountability and reduce material consumption through controlled access \citep{falasca2016performance, KROS2019100506}. Furthermore, employee efficiency is improved by placing IVMs near the point of work \citep{Manrique2015TheGS}. These benefits collectively lead to reduced inventory costs, minimized downtime, and improved operational efficiency.

IVMs come in several types (see Figure~\ref{fig:vending_machines_1}): coil machines dispense individual items, offering control over fast-moving items, and allow product-specific access controls. Sensor machines provide access with item traceability. Options include lockers, cabinets, and drawers that support various sizes of products. Locker machines vend large items, manage asset check-out/return, and serve as pick-up points for deliveries. Each type is equipped with software that integrates with the company's inventory management system, allowing for real-time tracking and automated reordering processes~\citep{1fastvend_nodate}.

\begin{figure}[htp]
\captionsetup{font=small} 
\FIGURE
{\begin{subfigure}[b]{0.3\textwidth}
    \includegraphics[width=\textwidth]{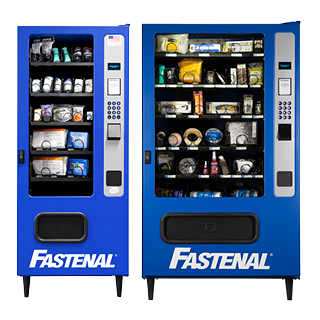}
    \caption{Coil Machine}
\end{subfigure}
\hfill
\begin{subfigure}[b]{0.3\textwidth}
    \includegraphics[width=\textwidth]{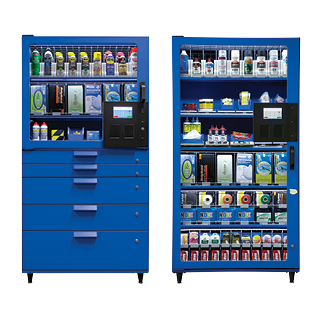}
    \caption{Sensor Machine}
\end{subfigure}
\hfill
\begin{subfigure}[b]{0.3\textwidth}
    \includegraphics[width=\textwidth]{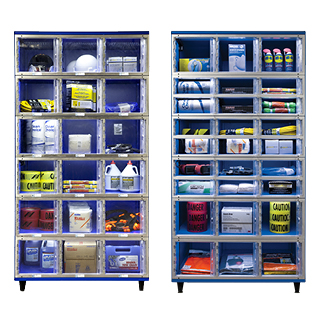}
    \caption{Locker Machine}
\end{subfigure}}
{Types of Industrial Vending Machines\label{fig:vending_machines_1}}
{Source: \citet{1fastvend_nodate}}
\end{figure}

Despite the ubiquitous use of IVMs with these advanced capabilities in the industry, there is a notable gap in research regarding optimal inventory management and replenishment strategies in IVM contexts. Although IVMs provide sophisticated technological features, their full potential can only be realized through the development and implementation of effective inventory control mechanisms. The current literature has not adequately addressed how to optimally utilize real-time data and automated features of IVMs to minimize costs while maintaining high service levels. This work aims to develop a theoretical basis for optimal replenishment control in this new context that has become so prevalent in modern industry. By establishing a rigorous mathematical framework, we seek to provide insights that can be applied to the broader field of IVM-based inventory management. 

We demonstrate our solution approach on real-life data obtained from our industrial partner,
a leading distributor of industrial and construction supplies in North America 
with more than 3000 market locations and more than 100,000 IVMs deployed in customer locations. 
The company’s extensive network of IVMs poses a complex inventory management challenge, balancing high service levels with minimal operating costs.~\citep{Manrique2015TheGS}. 

Industrial distribution companies typically serve their customers from a local distribution center. The customer's role is primarily to consume items as needed, while the industrial distribution company (which we will refer to as ``supplier'' in the rest of the paper) manages all aspects of inventory replenishment. Typically, each customer is served by an assigned staff member who is in charge of supplying the IVMs at a customer site with a predetermined schedule. The cabin capacity of the trucks is sufficient to fully restock all of the IVMs at a customer site. The number of unique items supplied depends on the particular context, but is typically in the order of hundreds. Inventory placed in the IVMs are immediately paid for by the customer. Hence, the supplier does not carry inventory holding costs in the traditional sense but owns the IVMs and therefore incurs capital investment costs associated with placing them in the customer's site. 

The industry places a great deal of emphasis on customer service and targets very high service levels when setting inventory control policies. For example, our industrial partner currently visits a lot of its customers daily to restock IVMs, although they would like to reduce the replenishment frequency per week to limit operational costs. Hence, the problem is characterized by multiple items with stochastic demand, a fixed cost for each replenishment visit, and item-specific stockout costs, although several characteristics make it strictly different from the previously studied inventory control problems. The IVMs placed at the customer site have limited capacity and real-time inventory information is available. It is also reasonable to consider a positive, although relatively short (e.g., one day), lead time between the replenishment decision and the actual replenishment at the customer site.

Since it is an important and nontrivial building block, we focus on the problem of optimizing the replenishment strategy for the IVMs at a single customer site, to minimize the long-run expected total cost per unit of time, where total cost includes fixed replenishment costs as well as stockout costs incurred when the customer cannot find an item in the IVM. More specifically, we are interested in determining (1) given IVM slot allocations for each item, when (or how frequently) should the supplier replenish the stock, and (2) developing scalable approximation methods that can provide near-optimal solutions for realistic instances with hundreds of items.

In the next section, we demonstrate that this introduces a new class of non-trivial replenishment control problems, despite the extensive existing literature on related problems. We formulate the control problem and present solution methodologies to obtain two different types of policies: (i) {\em trigger set control policy}, that is, policies that trigger a replenishment decision based on the state of the system (i.e., inventory levels of all items), and (ii) {\em fixed cycle replenishment}, that is, policies that replenish with fixed frequency. We offer theoretical analysis of both classes of policies and demonstrate their effectiveness through both a comprehensive numerical study and a real-world case study using our industrial partner's real-life transaction data from the collection of IVMs at a customer site.

\subsection{Technical Contributions:}
Our main technical contributions include the development and analysis of two distinct approaches to optimize IVM replenishment: a state-dependent trigger set control policy and a fixed cycle replenishment policy. In addition to establishing structural results for optimal control policy, we develop a novel near-optimal approximate online control framework that, for the first time in the literature, enables solving real-sized problems with hundreds of items to optimize the long-run average cost of IVM replenishments.

\subsubsection{Optimal Trigger Set Control Policy:}
We first formulate the IVM replenishment problem as a Markov decision process that aims to identify the set of states that trigger replenishment to minimize the long-run average cost per unit time. This formulation becomes computationally intractable because of the exponential growth of the state space with the number of items carried on the IVMs. To overcome this challenge, we formulate the problem as an optimal stopping-rule problem and show a number of structural results starting with monotonicity, which allows us to offer a computationally efficient and near-optimal approximate online control policy capable of solving problems with hundreds of items. 

\subsubsection{Fixed Cycle Replenishment Policy:}
We derive closed-form expressions for the optimal fixed replenishment cycle length that minimizes long-run average costs per unit time. Under reasonable conditions on the fixed replenishment cost, we prove that there exists a unique optimal cycle length that balances fixed costs against increasing stockout costs. Although this policy is theoretically suboptimal compared to the state-dependent trigger set control policy, it offers significant practical advantages in implementation due to simpler scheduling and coordination.

\subsubsection{Approximate Online Control Framework:}
A key contribution of our work is solving the IVM replenishment problem at industrial scale. For inventory systems with multiple items and a joint fixed cost for replenishment, finding optimal policies is notoriously difficult - the computational requirements grow exponentially with the number of products, making traditional approaches intractable beyond five items \citep{melchiors2002calculating}. While \citet{ignall1969optimal} proved that optimal policies for such systems cannot generally be described by a simple set of parameters per item, we show that for the IVM context, where all products are replenished to capacity upon ordering, the optimal policy has a remarkably simple structure. 

Through our monotonicity results, we show that the entire optimal policy can be characterized by a single critical value, which can be efficiently approximated. Additionally, we derive closed-form expressions for optimal fixed cycle replenishment. Both approaches can handle industrial-scale problems, which we demonstrate through our case study by successfully implementing both fixed cycle and near-optimal trigger set control policies for more than 300 products with varying demand rates. The approximate online control framework achieves a cost reduction of \mbox{4.1 to 22.9\%} compared to optimal fixed cycle replenishments while adopting optimal fixed cycle replenishment frequencies shows potential cost reductions of \mbox{61.7 to 78.6\%} compared to current practices, which typically service customer locations daily.
Our numerical study demonstrates significant computational efficiency gains with our approximation methods. For a problem with only six items, calculating the optimal policy required nearly 75 seconds in our experiments. In contrast, our approximations required only 0.05 seconds or 0.0004 seconds for the same instance, depending on the approximate expression used. This efficiency gap widens further as the problem size increases, making our approximation methods essential for industrial-scale problems with hundreds of items.


\subsection{Organization}
The remainder of this paper is structured as follows. Section~\ref{sec:literature} presents a review of the literature, followed by a detailed description of the problem in Section~\ref{sec:problem}. In Section~\ref{sec:optimal_policy}, we formulate the optimal trigger set control policy problem and provide structural results. In Section~\ref{sec:periodic_policy}, we present fixed cycle replenishment and determine the optimal replenishment cycle time. 
Section~\ref{sec:numerical} presents a numerical study that compares the performance of different policies on various dimensions of the problem. Section~\ref{sec:casestudy} demonstrates the practical application of our methodologies through a real-world case study using 
actual transaction data from our industrial partner. Finally, Section~\ref{sec:Conclusions} offers concluding remarks and directions for future work.

\section{Literature Review}
\label{sec:literature}

The structure of VMI agreements typically falls into two categories: (i) supplier-owned  inventory (consignment), where the supplier retains ownership until the point of sale~\citep{ katariya2014cyclic}, and (ii) traditional VMI, where the retailer assumes ownership of the inventory upon delivery. To mitigate the risk of overstocking in traditional VMI arrangements, contracts often stipulate inventory upper bounds with associated penalties for non-compliance~\citep{fry2001coordinating,darwish2010vendor, hariga2013vendor}. In addition, stock-out penalties are commonly implemented to ensure balanced inventory management and maintain service levels~\citep{fry2001coordinating, hariga2013vendor}. IVMs, on the other hand, have an inherent upper bound on slot capacity, providing a built-in hard constraint on stock levels. We refer to~\citet{govindan2013vendor} and~\citet{js2019literature} for an in-depth review of VMI, and continue to review more closely related literature. 

\paragraph{Literature on vending machines:} Regular vending machines differ from IVMs primarily in ownership and cost structure. In B2C or retail vending machines, the business typically owns both the machine and its inventory and bears inventory holding costs. This affects inventory management, as traditional order-up-to policies may not be optimal when balancing stock levels for profitability.

In the industrial distribution domain, \citet{zhang2022min} represents one of the few studies specifically addressing inventory management for IVMs. They focus on determining a continuous review min-max $(s,S)$ inventory policy, for each item. However, it should be noted that, as demonstrated by \citet{ignall1969optimal}, even the best $(s,S)$ policy may not achieve optimality in a multi-item setting. In the medical sector, IVMs can take the form of automated dispensing systems (ADS) for medications. \citet{dobson2019reducing} propose an integer linear programming algorithm to optimize ADS configurations. 

In the domain of regular vending machines, several key studies have advanced our understanding of B2C vending inventory management. \citet{miyamoto2003algorithms} addressed slot allocation for vending machines, developing an integer linear programming model to determine the optimal slot allocation and replenishment cycles. \citet{park2013operation} expanded on this, exploring stockout-based substitution and integrating slot allocation, replenishment points, thresholds, and vehicle routing into a profit-maximizing heuristic approach. The rise of smart vending machines has introduced new opportunities for replenishment optimization. \citet{poon2010real} proposed a real-time replenishment system, using technology such as GPRS and WiFi. 

To our knowledge,~\citet{zhang2022min} presents the only study addressing inventory management for IVMs in a non-medical setting. The scarcity of research in this specific domain highlights a significant gap that our research aims to fill by introducing a novel set of formulations, solution methodologies, and a real-life case study.

\paragraph{Relevant inventory control theory literature:} The Joint Replenishment Problem with deterministic demand (DJRP) and stochastic demand (SJRP) are established areas of research that address the coordination of replenishments for multiple items to minimize total replenishment costs. These problems generally apply to two main contexts: one-product, multiple-location settings (e.g., a central warehouse supplying multiple retailers) and single-location, multiple-product scenarios. 

The key components of a joint replenishment problem include (i) a major ordering cost per replenishment (similar to our model's fixed replenishment cost, $A$) and (ii) a product-specific minor ordering cost for including product $j$ in a replenishment. In the deterministic case (i.e., DJRP), demand is assumed to be known and constant, while the stochastic case (i.e., SJRP) deals with uncertain demand.  For a more comprehensive review of the JRP and its variants, we refer the reader to the review articles~\citet{khouja2008review},~\citet{bastos2017systematic} and~\citet{peng2022review}. 

In the JRP domain, our work is more related to the SJRP. However, we consider item-specific stockout costs $b_j$, instead of minor ordering costs. Furthermore, following the retailer-owned VMI approach, we do not directly define inventory holding costs for the supplier. However, we acknowledge that the supplier incurs implicit holding costs due to the capital investment for the IVMs and reflect this fact through the use of limited IVM slot capacity in our formulations. More direct consideration of IVM investment costs is discussed in Section~\ref{sec:Conclusions} as a future area of research. 


Among continuous review policies, the can-order policy, known as $(s,c,S)$, has been a popular tool of researchers in this area~\citep{balintfy1964basic}. Under this policy, an order is triggered when an item's inventory position reaches its must-order point, $s_j$. Other items are included in the order if their inventory position is at or below their respective can-order points~$c_j$. All items in the order are replenished up to their respective~$S_j$ levels.
The can-order point~$c_j$, allows flexibility in replenishment decisions. It enables the system to avoid replenishing products that still have sufficient inventory, saving on minor replenishment and holding costs.

However, in terms of the optimality of the $(s,c,S)$ policy,~\citet{ignall1969optimal} demonstrated that can-order policies are not always optimal and that the optimal policy lacks a simple structure.  Here, a ``simple structure'' refers to policies that can be defined by a fixed set of parameters for each product independently, such as the $(s,c,S)$ policy with its must-order point $s_j$, can-order point~$c_j$, and order-up-to level~$S_j$ for each product~$j$. 
Instead, an optimal policy requires defining order up to levels~$S_j(\bm i)$ for each state~$\bm i$, which is a vector that represents the levels of inventory of all items. This state-dependent approach eliminates the need for can-order points, but significantly increases the policy's complexity.  In fact, even finding the best possible can-order policy becomes computationally intractable as the number of items increases~\citep{melchiors2002calculating}. However, due to its simple and intuitive structure, the literature has focused extensively on developing methods to find good can-order policies rather than optimal ones~\citep{silver1974control, federgruen1984coordinated, melchiors2002calculating, feng2015replenishment, creemers2022joint}.

Due to the lack of product-dependent holding and replenishment costs in our IVM replenishment problem context, we can show that it is optimal to replenish all products up to their slot allocation (i.e., the number of vending machine slots allocated to the storage of the item) whenever a replenishment occurs, making the order-up-to levels~$S_j$ constant across all states in which a replenishment is triggered. This differs from traditional VMI literature where determining optimal order-up-to levels is a primary focus~\citep{axsater2001note}. Consequently, our analysis focuses solely on determining the set of states that should trigger a replenishment. We approach this by defining a ``trigger set control policy,'' which specifies all inventory states that trigger a replenishment. This contrasts with traditional ``threshold policies'' that initiate a replenishment when the inventory of any single item reaches its individual threshold~$s_j$ in multi-item inventory management contexts. 

Figure~\ref{fig:optimal_against_threshold} illustrates an example showing that a threshold policy is not necessarily optimal for our IVM replenishment problem. The figure on the left shows the optimal policy (which is of trigger set type), and the figure on the right shows the best threshold policy for a simple, two-item example.  Since the triggering of a replenishment under an optimal policy for our problem depends on the inventory levels of all products in the IVM (i.e., the complete state variable), we do not assume an independent threshold policy structure when searching for an optimal replenishment policy. 

\begin{figure}
\FIGURE
{\includegraphics[width=0.6\textwidth]{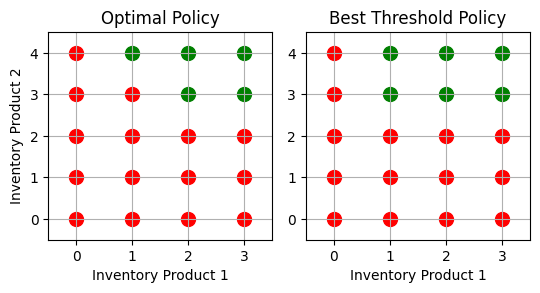}}
{Comparison of the optimal policy and best threshold policy\label{fig:optimal_against_threshold}}
{The optimal policy achieves an average cost of 7.3831 per unit time, while the best threshold policy has an average cost of 7.3943 per unit time. Replenishment happens in states marked with red dots, defined as {\it stopping states}. Parameters: $A = 10$ (fixed cost per replenishment), $\bm{i}^0 = (3, 4)$ (initial inventory state), $\bm{b} = [6, 6]$ (stockout costs for each item), $\bm{\lambda} = [1, 2]$ (Poisson demand rates), and $\tau = 1$ (lead time).}
\end{figure}

This is significant as our problem diverges from many continuous review policy papers that use threshold policies~\citep[such as][]{creemers2022joint, zhang2022min, minner2007replenishment, feng2015replenishment, park2012heuristic}. 
The popularity of these threshold policies can be explained by the complexity of finding an optimal policy and concerns about the practicality of implementation. We note that the latter concern is significantly alleviated in the IVM setting.

\citet{schulz2024integer} demonstrate that even optimizing the cycle length in the two-item Joint Replenishment Problem (JRP) is as difficult as integer factorization, underscoring the broader challenge of solving joint replenishment problems optimally. \citet{ignall1969optimal} further explain that the optimal policy can be determined using a renewal Markov Decision Process (MDP). However, the state and action spaces grow exponentially with the number of items making the renewal MDP impractical to solve for more than a few items. Recent research has sought to mitigate this curse of dimensionality. For instance, \citet{zhang2024low} propose a low-rank approximation via moment coupling, which significantly reduces the state space size in MDPs. However, reduction remains limited, making exact solutions for realistic JRPs infeasible and reinforcing the need for heuristics and structural insights in designing practical policies. 
A related challenge appears in ATM replenishment, where the problem is which ATMs to replenish. Like JRP, it faces computational challenges due to the exponential growth of state space. To address this, \citet{zhang2018automated} develop an index-based policy, assigning indices to ATMs to guide replenishment decisions. This parallels our approach, where indices derived from inventory states determine optimal replenishment states.

Our approach involves translating this complex Markov renewal process problem into an optimal stopping problem with useful structural properties (monotonicity). To solve this optimal stopping problem, we provide a linear programming formulation, although the problem of exponential state growth still renders it impossible to solve realistic scenarios in practical time. Fortunately, the structural properties shown below yield an approximate online policy that can be used to solve real instances of the problem almost optimally. 


\section{Problem Description}
\label{sec:problem}

We consider one or more IVMs located at a single customer site, with a total of $C$ slots to store inventory of $J$ unique items.  The number of storage slots allocated to each item $j$ is denoted by ${Q_j \in \mathbb{Z}}$, where ${Q_j > 0}$ and ${\sum_j Q_j=C}$.  Each item ${j \in J}$ is characterized by random demand with a rate of $\lambda_j \in \mathbb{R}$ per unit time, and demand arrivals for each item $j$ are modeled using a Poisson process that is independent of the demand for other items. 

A stockout cost of $b_j$ is incurred for each item $j$ demanded by the customer but unavailable to dispense through the IVM. All stockouts are recorded and treated as lost sales to the IVM; we assume that these demands are met through other means, which is included in the stockout cost~$b_j$. The fixed replenishment cost,~$A$, is constant and independent of the item mix and quantity that is restocked in the IVM(s) during the replenishment event. The supplier is paid in full for all inventory placed in the IVM at the time of replenishment.

Let $\tau$ denote the known and constant lead time needed to resupply the IVMs on site (that is, time to load, transport, restock, etc.), once a replenishment action is triggered. We consider that the system can be replenished at any time during working hours at the customer site and that there are no limitations or additional costs associated with transporting any quantity of items to the customer site.
We further characterize the replenishment process following the current industrial practice. 

\begin{assumption}\label{assum:no_overlap} 
A replenishment to the customer site can accommodate any quantity of items required to restock the machine in a single trip. In addition, there is at most one outstanding replenishment at any time.
\end{assumption}

The first part means that the driver assigned to the customer site can accommodate any quantity of items required to restock the machine in a single trip. The latter part follows because there is one assigned driver to a site, eliminating the possibility of more than one outstanding order. These assumptions are typically well justified, based on our work with our industrial partner. 

The system state at any given time is assumed to be known with perfect information by the supplier (owing to IVM capabilities) and is represented by the vector, $\bm {i} = (i_1, i_2, \ldots, i_J)$, where $i_j\leq Q_j$ denotes the current inventory level of item~$j$.  We use $\bm{Q} = (Q_1, Q_2, \ldots, Q_J) \in \mathbb{Z}^{J}$ to represent the vector of order-up-to levels, which is the initial state in each replenishment cycle, $\bm{i^0}$.
For a complete list of notation used throughout the paper, we refer the reader to Table \ref{tab:notation} in the Appendix \ref{sec:notation}.

We would like to identify an optimal control policy that triggers replenishments on the basis of this system state, to minimize the long-run expected average cost per unit time. 
A simple characterization for this context is that there is no advantage to not filling all slots allocated to the item $j$ when visiting the customer site to replenish other items. 


\begin{proposition}
The following properties hold for an optimal policy: (i) each time the customer site is visited, all items are restocked to their respective order-up-to levels (i.e., $Q_j$ for item $j$), (ii)~there is at most one outstanding replenishment order at any time.
\end{proposition}

\begin{proof}
  {\it Proof.}  The result in (i) follows from two key assumptions: the absence of holding costs and the lack of replenishment capacity limits. Without holding costs, there is no penalty for maintaining high inventory levels, given $Q_j$ for all items $j$. Furthermore, without constraints on replenishment quantities, it is always feasible to restock each item to its maximum capacity. Consequently, filling all items to their $Q_j$ levels at each visit minimizes the risk of stockouts between visits without incurring additional costs. For the property in (ii), consider a state where a replenishment order is outstanding. Since this order will restore all items to their maximum capacity $Q_j$ after lead time $\tau$, triggering another replenishment before the first one arrives cannot reduce stockout costs - the second order cannot arrive before the first one completes. Moreover, triggering the second order would incur an additional fixed cost $A$. Therefore, it is optimal to wait until the first replenishment is complete before considering another replenishment decision. \qed
\end{proof}

\section{Trigger Set Control Policies: Analysis and Implementation}
\label{sec:optimal_policy}

In this section, we characterize the structure of optimal trigger set control policies and develop computationally efficient approximation methods to identify these policies. We first present a continuous-time Markov chain model and derive the average cost expression under a trigger set, allowing us to formulate an LP to find the optimal trigger set. Although this provides an optimal solution, its implementation is impractical for large-scale applications. Therefore, we formulate a stopping rule problem, which we use to prove the monotonicity of an optimal policy. Finally, we show how these theoretical properties lead to an online control policy framework and develop computationally efficient approximations of this policy for practical implementation.

\subsection{Binary Decision Renewal Process Formulation}
\label{Model Formulation}
We model the dynamics of system state $I(t)$ as a continuous-time Markov chain where the time between events (that is, demand arrivals) is exponentially distributed with rate $\Lambda = \sum_{j=1}^J \lambda_j$, transitions between states are governed by the exponential arrival rates of the items, $\lambda_j$, and the initial state, $I(0)=\bm {i^0}$. 
At state $\bm i$, a demand for item $j$ decreases $i_j$ by 1, and the system transitions to state $\bm{i^{-j}}=(i_1,\dots,i_j-1,\dots,i_J)$. If $i_j \leq 0$ when a demand occurs, a stockout is recorded and a stockout penalty of $b_j$ is incurred for each unit of item $j$ that is demanded but unavailable in the IVM.

Replenishment orders to restock all items to their respective $Q_j$ levels are triggered when the system enters one of the states in the {\it trigger set}. Upon restocking, the state is returned to $\bm{i^0}$. Hence, the system dynamics constitutes a regenerative process that probabilistically restarts itself. 

Within the state space $\cal{S}$, we define the trigger set $\cal U$ as the subset of states where replenishment is triggered and $\cal{\overline{\cal U}}$ as the set of continuation states where replenishment is not triggered. Naturally, these sets are complementary, that is, $\cal U \cup \overline{\cal U} = \cal{S}$ and $\cal U \cap \overline{\cal U} = \emptyset$. We consider an infinite-horizon control problem, where the objective is to find the trigger set that minimizes the long-run expected average total cost per unit time. 

\begin{assumption}\label{assum:initial_state}
  We impose the following constraints on our model: (i) $\bm{i^0} \in \cal{\overline{\cal U}}$, i.e., we do not trigger when the IVM is full, and (ii) there exists at least one state reachable in finite time that triggers replenishment, i.e., $\cal{U}_r \subset \mathcal{S}$ and $\cal{U}_r \neq \emptyset$. This ensures that the system is replenished after a finite amount of time, $\tau$ and returns to the state $\bm{i^0}$, maintaining ergodicity.
\end{assumption}

The role of the assumption is examined in Section~\ref{sec:numerical}, where we establish bounds on the fixed cost,~$A$ that ensure that both assumptions are satisfied. In particular, we show how a minimum fixed cost $A_{\min}$ ensures the initial state is in the continue set (Assumption~\ref{assum:initial_state}) and prevents overlapping replenishment orders (Assumption~\ref{assum:no_overlap}).

We are now ready to state the binary-decision renewal process formulation. At each decision epoch, the supplier chooses between two actions: (1) stop and trigger a restock of the IVM or (2) continue. 
Let $G(\bm{i})$ denote the expected immediate cost of stopping in state $\bm{i} =(i_1, i_2, \dots, i_J) \in \cal S$. The following expression for $G(\bm{i})$ accounts for the fixed cost, as well as the stockout costs already incurred (if $i_j <0$) and may be incurred during the replenishment lead time. Hence, $G(\bm{i})$ accounts for all the costs incurred in a cycle, and the cost of continuing in any state $\bm i$ is zero. 
\begin{equation}
    G(\bm{i}) = A + \sum_{j=1}^J \; b_j \sum_{r=i_j}^\infty  \left(r-i_j\right) P (D_j(\tau)=r)~,  \quad \forall \bm{i} \in \cal U~,
\end{equation}
where $D_j(\tau)$ denotes the random demand of item $j$ during the lead time of $\tau$ time units, and naturally, $P(D_j(\tau)=r)=0$ for all $r<0$. Since demand for each item $j$ follows a Poisson process with rate $\lambda_j$, $D_j(\tau)$ is Poisson with rate $\lambda_j \tau$.

The continuous time Markov chain evolves according to the following dynamics. With the initial state of $\bm{i^0}$, the system state transitions to $\bm{i^{-j}}$ every time a demand for item $j$ occurs. The time between transitions due to demand arrivals, by definition, is exponential with rate $\Lambda=\sum_{j=1}^J \lambda_j$. After a random number of $n$ transitions, the system enters a state $\bm{i^n} \in \cal U$, triggering a replenishment. During the lead time of $\tau$, no new replenishment decisions are made, so we maintain the state at $\bm{i^n}$ while accounting for the stockout costs during this period through $G(\bm i)$. After $\tau$ time units, the state transitions to $\bm{i^0}$ upon restocking. 

Let $P(\bm m|{\bm i})$ denote the transition probabilities of the uniformized chain (with uniformization rate of~$\Lambda$); transitions from continue states $\bm i$ are only possible to states with one fewer inventory of one of the $J$ items. Arrival of demand for item $j$ transitions the system to state ${\bm i}^{-j}=(i_i, i_2, \dots, i_j-1, \dots, i_J)$, which occurs with probability $p_j = \lambda_j/\Lambda$. 

\begin{restatable}{lemma}{avgcostprop}\label{lemma:avg_cost}
For any trigger set $\cal U$, using uniformization with rate $\Lambda$, the long-run expected average cost per unit time is given by
\begin{eqnarray}
    \alpha_\cal U = \frac{\sum_{\bm i \in \cal {U}}G(\bm i)\pi_{\bm i}}{({1}/{\Lambda})\;\sum_{\bm i \in \cal {\bar U}} \pi_{\bm i}+\tau\;\sum_{\bm i \in \cal {U}}\pi_{\bm i}}~,\label{eq:long_term_avg_cost}
\end{eqnarray}
where the stationary distribution, $\pi_{\bm i}$ for all $\bm i \in \cal S$ can be obtained by solving 
\begin{eqnarray}
\displaystyle
\pi_{\bm m} = \delta(\bm m,\bm Q) \sum_{\bm i \in \cal U} \pi_{\bm i} + \sum_{\bm i \in \overline{\cal U}} {P}(\bm m|\bm i) \pi_{\bm i} ~, \quad \text{ and } \quad 
\sum_{\bm i \in \cal U} \pi_{\bm i} + \sum_{\bm i \in \overline{\cal U}} \pi_{\bm i} =1~.
\end{eqnarray}
The binary indicator function, $\delta(\bm m,\bm Q)=1$ if $\bm m= \bm Q$ and takes a value of zero otherwise. 
\end{restatable}

The proof of Lemma \ref{lemma:avg_cost}, which derives the long-run expected average cost for a given trigger set, is provided in the Electronic Companion \ref{sec:Proofs for Trigger Set Control Policy}.

We now have a well-defined problem that we can formulate as a linear program (LP) to obtain the optimal trigger set that minimizes average cost, $\alpha$, since
\begin{eqnarray}
    \alpha^*= \min_{\cal U \subset \cal S} \; 
    \frac{\sum_{\bm i \in \cal {U}}G(\bm i)\pi_{\bm i}}{({1}/{\Lambda})\;\sum_{\bm i \in \cal {\bar U}} \pi_{\bm i}+\tau\;\sum_{\bm i \in \cal {U}}\pi_{\bm i}}~.\label{eq:objfn}
\end{eqnarray}
To linearize this objective function, we introduce variables 
$r_{\bm i}=\pi_{\bm i} \left ( \frac{1}{\Lambda}-(\tau-1) \sum_{\bm m \in \cal U} r_{\bm m} \right ), \forall \bm i \in \cal S$. 

\begin{restatable}{lemma}{linearization}
\label{lemma: Linearization}
By replacing $\pi_{\bm i}$ with $r_{\bm i}$ in our objective function, the nonlinear objective function simplifies to a linear form, given by $\Lambda \sum_{\bm i\in U} G(\bm i) r_{\bm i}$.
\end{restatable}
The proof for Lemma~\ref{lemma: Linearization} is provided in the Electronic Companion~\ref{sec:Proofs for Trigger Set Control Policy}.
The average cost per unit time is equal to $\Lambda$ times the average cost per period, as defined by the uniformized process. To obtain the optimal stopping set, $\cal {U^*}$, we formulate the following LP model.

\begin{mini}
{}{\Lambda\; \sum_{\bm i}G(\bm i) \,z_{\bm i}}{}{}
\addConstraint{z_{\bm m}+y_{\bm m} \leq \delta({\bm m},\bm Q)\,\sum_{\bm i} z_{\bm i} + \sum_{\bm i} P({\bm m}|{\bm i})\,y_{\bm i}~,}{}{\mkern53mu {\bm m} \in \cal S}
\addConstraint{\tau \Lambda \sum_{\bm i} z_{\bm i} + \sum_{\bm i} y_{\bm i} =1}{}{}
\addConstraint{y_{\bm i} \geq 0 \;\text{ and } \; z_{\bm i} \geq 0~, \quad {\bm i} \in \cal S~,}{}{\mkern53mu } \label{eq:LP_Formulation}
\end{mini}
where $z_{\bm i}$ and $y_{\bm i}$ are decision variables corresponding to linearization variable $r_{\bm i}$ defined above:  $z_{\bm i}$ equals $r_{\bm i}$ when state $\bm i \in \cal U$, and $y_{\bm i}$ equals $r_{\bm i}$ when $\bm i \in \overline{\cal U}$. Specifically, $z_{\bm i}$ is positive if and only if state $\bm i$ is in the trigger set, and $y_{\bm i}$ is positive if and only if state $\bm i$ is in the continue set. For each state $\bm i$, at most one of $z_{\bm i}$ or $y_{\bm i}$ can be positive.  The LP formulation has $2|\mathcal{S}|$ number of variables and $|\mathcal{S}|+1$ number of constraints.

Although this LP exactly calculates $\cal U^*$, it cannot be used to solve real-life problems. To illustrate the computational challenge involved with solving the problem, consider a modest-sized instance with 10 items, each having a slot capacity of 15.  For this example, considering only nonnegative inventory levels (0 to $Q_j$), the size of the state space would be $|\mathcal{S}| = 16^{10} \approx 1.1 \times 10^{12}$ states, requiring more than a trillion variables and constraints.  Even with modern computing power and sophisticated LP solvers, problems of this size are computationally intractable, highlighting the need for a more efficient solution approach.

\subsection{Stopping-rule Problem Formulation}
\label{sec:stoprule}

We next formulate the problem as a stopping rule problem and show monotonicity of an optimal policy to offer a near-optimal online control framework capable of solving real-life problems. 

By the renewal reward theorem, the long-run expected average cost per unit time under a stopping set~$\cal U$ can be expressed as the ratio of the expected costs in a cycle to the expected length of a cycle. The optimal stopping set,~$\cal U^*$ is the one that minimizes the long-run expected average cost, that is, 
\begin{eqnarray}
\alpha^* = \frac{\mathbb{E}_{\cal U^*}[\text{cost per cycle}]}{\mathbb{E}_{\cal U^*}[\text{length of cycle}]}  \leq \frac{\mathbb{E}_{\cal U}[\text{cost per cycle}]}{\mathbb{E}_{\cal U} [\text{length of cycle}]} ~, \quad \forall \; \cal U \subset \cal S~. \label{alpha_star}
\end{eqnarray}
Let ${TC}$ and ${T}$ denote the random cost per cycle and length of cycle, respectively. The above implies that the optimal expected average cost per unit time $\alpha^*$ satisfies
$\mathbb{E}_{\cal U}[-{TC} +  \alpha^* \; {T}] \leq 0$ for all $\cal U \subset \cal S$. 
Unlike the average cost per unit time criterion, this expression can be represented by a stopping rule problem, where the optimal trigger set $\cal U^*$ achieves $\mathbb{E}_{\cal U^*}[-{TC} +  \alpha^* \; T] = 0$. Expanding the terms inside the expectation for trigger set~$\cal U$, we obtain
\begin{eqnarray}
-{TC} +  \alpha^* \; T 
&=&  -G(\bm {i^n}) + \alpha^* \;\left (\tau + \frac{n}{\Lambda} \right )= (-G(\bm {i^n}) + \alpha^* \; \tau )+ \alpha^* \;  \frac{n}{\Lambda}~,
\end{eqnarray}
where $n$ denotes the random number of steps that the uniformized chain takes to reach a state that is in the trigger set $\cal U$. 
Let us now consider a stopping problem defined by a set of required stopping states, $\cal T_r$, a set of required continue states, $\cal T_c$, the reward received upon stopping at state $\bm i$ of $F(\bm i)$ and the cost incurred for continuing one more step from state $\bm i$, $f(\bm i)$. As an instance of this stopping rule problem, consider 
$\cal T_r = \cal U_r$, $\cal T_c = \{\bm Q\}$, $F(\bm i) = - G(\bm i)+ \alpha^*\; \tau$, and $f(\bm i) = - \alpha^* \; ({1}/{\Lambda})$. 
If we use $\cal U$ as a stopping set for this problem, the expected total payoff, $\mathbb{E}_{\cal U} Z=\mathbb{E}_{\cal U} [-{TC} +  \alpha^* \; T]\leq 0$. For $\cal T^*=\cal U^*$, $\mathbb{E}_{\cal T^*} [-{TC} +  \alpha^* \; T] = 0$, and $\cal U^*$ is a solution to this problem. For a formal statement of the result, please see Theorem 10.5 of \citep{breiman1964stopping}. 

\subsubsection{Transformation to an Entrance Fee Problem: }

The stopping rule problem involves a cost of continuing and reward of stopping. This problem can be transformed into an {\it {entrance fee}} problem, which assumes that an entrance fee is incurred each time the decision is to continue, but there is no cost or reward of stopping. This problem is defined by the same required stop and continue sets, and with the entrance fee at state $\bm i$, given by 
\begin{eqnarray}
    f'(\bm i) = f(\bm i) - \left[\sum_j p_j F({\bm i^{-j}}) - F(\bm i)\right]~,
\end{eqnarray}
where $p_j = {\lambda_j}/{\Lambda}$ is the probability of demand for item $j$, and ${\bm i^{-j}}$ denotes the state with one fewer item~$j$ compared to state $\bm i$. 
The entrance fee $f'(\bm i)$ represents the difference between the cost of continuing one step, $f(\bm i)$, and the expected change in stopping reward, $\sum_j p_j F({\bm i^{-j}}) - F(\bm i)$.
Therefore, it follows that in state $\bm i$, continuing is optimal if the expected improvement in stopping reward exceeds the cost of continuing, that is, when $f'(\bm i) \leq 0$.
    
\begin{restatable}{lemma}{entrancefeeexpression}\label{lemma:entrance_fee_expression}
The entrance fee $f'(\bm i)$ can be expressed as
\begin{eqnarray}
f'(\bm i) = - \alpha^* \; \frac{1}{\Lambda} + \sum_j p_j b_j P(D_j(\tau) \geq i_j)~.
\end{eqnarray}
\end{restatable}
The proof for Lemma \ref{lemma:entrance_fee_expression} is provided in the Electronic Companion~\ref{sec:Proofs for Trigger Set Control Policy}.

\subsubsection{Monotonicity of the Optimal Replenishment Policy: }

In our IVM replenishment problem, monotonicity refers to a property that allows us to index the states such that there exists a critical state such that it is optimal to trigger a replenishment for all states with an index greater than that of the critical state. Theorem~\ref{thm:threshold} establishes this property of an optimal policy. 

\begin{definition}
Let $\kappa: \mathcal{S} \rightarrow {1,2,\dots,|\mathcal{S}|}$ be a mapping between 
a state vector $\bm i \in \mathcal{S}$
and its index $\kappa(\bm i)$, which indicates the state's ranking with respect to the entrance fee, $f'(\bm i)$. The inverse mapping $\kappa^{-1}(k)$ denotes the state vector with the $k$th smallest entrance fee, $f'(\cdot)$.
\end{definition}

\begin{restatable}{theorem}{threshold}
\label{thm:threshold}
Indexing the states such that $f'(\cdot)$ is non-decreasing in the state index $k$, there exists a critical state index $k^*$ such that it is optimal to trigger replenishment at all states with indices greater than or equal to $k^*$, and to continue for all states with indices less than $k^*$.  
\end{restatable}
The proof of Theorem \ref{thm:threshold}, which involves demonstrating that the entrance fee problem satisfies the conditions of monotone optimal stopping rules, is provided in the Electronic Companion \ref{sec:Proofs for Trigger Set Control Policy}.

Theorem \ref{thm:threshold} ensures monotonicity. That means, once we hit a state where continuing is unfavorable i.e., $f'(\bm i) > 0$, all subsequent states must also be unfavorable, stated by Corollary 1. 

\begin{restatable}{corollary}{optimalContSet}\label{cor:optimal_continue_set}
The optimal trigger set control policy that minimizes the long-term average cost has a continue set $\overline{\mathcal{U}}^*$ consisting of exactly those states where $f'(\bm i) \leq 0$, and a stopping set $\mathcal{U}^*$ consisting of states where $f'(\bm i) > 0$.
\end{restatable}

The proof of Corollary \ref{cor:optimal_continue_set} is provided in the Electronic Companion \ref{sec:Proofs for Trigger Set Control Policy}.

\subsection{Exact Online Optimal Control Policy}

Corollary~\ref{cor:optimal_continue_set} yields a simple and intuitive optimal policy structure that can be evaluated online. That is, the optimal decision can be determined using only local state information and a single policy parameter at each decision epoch. 

\begin{restatable}{theorem}{onlinepolicy}
\label{thm:online_policy}
The optimal policy, $\mu_{\alpha^*}$ can be characterized using the function 
\begin{eqnarray}
\hat{g}({\bm i}) = \sum_j \lambda_j b_j P(D_j(\tau) \geq i_j)~.
\end{eqnarray}
At state $\bm i \in \cal S$, the optimal action is defined by $\mu_{\alpha^*}(\bm i)$, where
\begin{eqnarray}
\mu_{\alpha^*}({\bm i}) = \begin{cases}
\text{continue} & \text{if } \hat{g}({\bm i}) \leq \alpha^* \\
\text{stop} & \text{if } \hat{g}({\bm i}) > \alpha^*
\end{cases}
\end{eqnarray}
with average cost of $c(\mu_{\alpha^*}) = \alpha^*$ per unit time.
\end{restatable}

The proof of Theorem \ref{thm:online_policy} is provided in the Electronic Companion \ref{sec:Proofs for Trigger Set Control Policy}.

For each state $\bm i$, we define $\mu_{\hat{g}(\bm i)}$ as the policy that continues in all states $\bm m$ with ${\hat{g}(\bm m)\leq \hat{g}({\bm i})}$ and stops otherwise, and denote its long-term average cost by $c(\mu_{\hat{g}(\bm i)})$, which can be computed by solving for the stationary distribution of the resulting Markov chain and calculating the corresponding average cost per unit time using Equation~(\ref{eq:long_term_avg_cost}). 
The monotonicity property established in Theorem~\ref{thm:threshold} ensures that $\alpha^*$ is indeed the global minimum. This follows because the monotonicity property proves that the optimal policy must be of the form ``continue in all states with $\hat{g}(\cdot)$ below some critical value,'' and we evaluate all possible such critical values when computing $c(\mu_{\hat{g}(\bm i)})$ for each state $\bm i$.


Using Equation~(\ref{eq:long_term_avg_cost}) to calculate the long-run average cost each time we add a new state to the continue set is computationally inefficient. This naive approach requires repeatedly solving a large Markov chain problem from scratch after each state addition.
We now show that this computational complexity can be avoided.

\begin{restatable}{lemma}{incrementalcalc}\label{lemma:incremental_calc}
Let $\rho_{\bm i}$ denote the probability of visiting state $\bm i$ during a cycle starting from state ${\bm {i^0}}$. When adding states to the continue set $\overline{\cal U}$ in ascending order of $\hat{g}(\cdot)$, for each newly added state $\bm i$, its visiting probability $\rho_{\bm i}$ can be directly calculated as
\begin{equation}
\rho_{\bm i} = \frac{n_{\bm i}!}{\prod_{j=1}^J \Delta_{i_j}!} \prod_{j=1}^J p_j^{\Delta_{i_j}}, \quad \forall \bm i \in \cal S,
\label{eq:rho_calc}
\end{equation}
where $\Delta_{\bm i} = (\Delta_{ i_1}, \ldots, \Delta_{ i_J}) = ({i^0}_1 - i_1, \ldots, {i^0}_J - i_J)$ is the vector of inventory reductions from state $\bm {i^0}$ to state $\bm i$, $n_{\bm i} = \sum_{j=1}^J \Delta_{ i_j}$ is the total inventory reduction, and $\bm{p} = (\lambda_1/\Lambda, \ldots, \lambda_J/\Lambda)$ is the vector of normalized demand rates.
\end{restatable}
The proof of Lemma~\ref{lemma:incremental_calc} is provided in the Electronic Companion~\ref{sec:Proofs for Trigger Set Control Policy}.
Once the visiting probabilities, $\rho_{\bm i}$ are known, it is possible to determine the expected cycle length and cost incrementally, relying solely on the continue states. 

\begin{restatable}{theorem}{CycleDecompTheoremName}
\label{thm:cycle_decomp}
Under the transformation from the original stopping problem to an entrance fee problem, the expected cost per cycle $\mathbb{E}[TC]$ and the expected cycle length $\mathbb{E}[T]$ can be expressed solely in terms of continue states. This allows us to calculate the average cost per unit time as
\begin{eqnarray}
\alpha_{\overline{\cal U}} &=& \frac{\mathbb{E}_{\overline{\cal U}}[TC]}{\mathbb{E}_{\overline{\cal U}}[T]} = \frac{G({\bm {i^0}}) + \sum_{i \in \overline{\cal U}} \frac{\hat{g}(\bm i)}{\Lambda} \rho_{\bm i}}{\tau + \sum_{i \in \overline{\cal U}} \frac{1}{\Lambda} \rho_{\bm i}} \label{eq:long_term_avg_cost2}.
\end{eqnarray}
\end{restatable}
The proof for Theorem \ref{thm:cycle_decomp} is provided in Electronic Companion~\ref{sec:Proofs for Trigger Set Control Policy}.
This theorem provides a computationally efficient way to update the long-term average cost when adding an additional state $\bm i$ to the continue set $\overline{\cal U}$. When we add a state $\bm i$ to the continue set $\overline{\cal U}$, it contributes $\frac{\hat{g}(\bm i)}{\Lambda} \rho_{\bm i}$ to the expected cost per cycle and $\frac{1}{\Lambda} \rho_{\bm i}$ to the expected cycle length. 

Algorithm \ref{alg:continue_set} implements this approach. For each state, we calculate its visiting probability $\rho_{\bm i}$ using Equation~\eqref{eq:rho_calc}, update the total expected cost and cycle length, and then calculate the new average cost. We keep adding states as long as the average cost decreases. Table~\ref{table:state_info} demonstrates the algorithm using our example parameters, showing how the optimal continue set is built incrementally until reaching state (1,4), where the minimum average cost of $\alpha^*=7.3832$ is achieved.

\begin{figure}[htbp]
\centering
\begin{minipage}[c]{0.48\textwidth}
        \begin{algorithm}[H]\small
                \captionsetup{font=small} 
            \caption{Incremental Construction of Continue Set}
            \label{alg:continue_set}
            \begin{algorithmic}[1]
                \baselineskip=0.75\baselineskip  
                \State Sort all states by increasing $\hat{g}(\cdot)$ values
                \State Initialize $\overline{\cal U} = \emptyset$, $TC = G({\bm {i^0}})$, $T = \tau$
                \For{each state $\bm i$ in sorted order}
                    \State $TC_{new} = TC + \frac{\hat{g}(\bm i)}{\Lambda} \rho_{\bm i}$, $T_{new} = T + \frac{1}{\Lambda} \rho_{\bm i}$
                    \State $c(\mu_{\hat{g}(\bm i)}) = \frac{TC_{new}}{T_{new}}$
                    \If{$c(\mu_{\hat{g}(\bm i)}) < \hat{g}(\bm i)$}
                        \State Add $\bm i$ to $\overline{\cal U}$
                        \State $TC = TC_{new}$, $T = T_{new}$
                    \Else
                        \State \textbf{break}
                    \EndIf
                \EndFor
            \end{algorithmic}
        \end{algorithm}
\end{minipage}
\hfill
\begin{minipage}[c]{0.48\textwidth}
\begin{table}[H]
\TABLE
{States ordered by entrance fee $f'({\bm i})$ with parameters: $A=10$, $\tau=1$, ${\bm b}=[6,6]$, $\bm{\lambda}=[1,2]$, and initial state ${\bm i^0}=(3,4)$\label{table:state_info}}
{    \begin{tabular}{|c|c|c|c|c|c|}
        \hline
        ${\bm i}$ & $\hat{g}({\bm i})$ & $\rho({\bm i})$ & $TC_{\text{new}}$ & $T_{\text{new}}$ & $c(\mu_{\hat{g}(\bm i)})$ \\
        \hline
        (3,4) & 2.1963 & 1.0000 & 11.3230 & 1.3333 & 8.4922 \\
        (2,4) & 3.3000 & 0.3333 & 11.6896 & 1.4444 & 8.0928 \\
        (3,3) & 4.3617 & 0.6667 & 12.6589 & 1.6667 & 7.5953 \\
        (2,3) & 5.4653 & 0.4444 & 13.4686 & 1.8148 & 7.4215 \\
        (1,4) & 5.5072 & 0.1111 & 13.6726 & 1.8519 & $\alpha^*=$7.3832 \\
        (3,2) & 7.6097 & 0.4444 & 14.7999 & 2.0000 & 7.4000 \\
        \hline
    \end{tabular}}
{}
\end{table}
\end{minipage}
{}
\end{figure}

This approach significantly simplifies calculations by focusing exclusively on the continue states, eliminating the need to identify the corresponding trigger states. It also avoids solving the Markov decision process for stationary distribution $\pi$ and instead uses straightforward incremental calculations at each step.  As shown in Appendix~\ref{app:comp_performance}, Algorithm \ref{alg:continue_set} also dramatically outperforms the LP formulation defined in Equation~(\ref{eq:LP_Formulation}) in computational efficiency.

However, while the algorithm itself is computationally efficient once states are ordered, this ordering becomes infeasible for realistic problems since the state space grows exponentially with the number of products.
In such cases, we cannot explicitly define the optimal continue set. Consequently, we also cannot determine the exact value of $\alpha^*$ required by Theorem~\ref{thm:online_policy} to define the optimal continue set. Therefore, in the next section, we introduce an approximation approach to estimate $\alpha^*$ which can then be used for the online policy defined in Theorem~\ref{thm:online_policy}.

\subsection{Approximate Online Control Framework}
\label{sec:ApproximationMethods}

To implement the online policy, instead of computing the complete continue and stopping sets — which would be prohibitively large in realistic settings — we only need a policy parameter $\alpha^*$. The system decides whether to continue or stop comparing the current state’s $\hat{g}(\cdot)$ value to this parameter, continuing if it is lower and stopping if it is higher.

While the online policy using $\alpha^*$ is optimal, computing the exact value of $\alpha^*$ is computationally intractable for realistic-sized problems. However, we can exploit the policy structure by approximating $\alpha^*$ while maintaining the simple parameter-based decision rule.  To this end, we offer two expressions to estimate $\alpha^*$. Both approximations avoid computational complexity by evaluating the state-based costs at an expected state. 

Our first approximation $\alpha_G$ approximates $\alpha^*$ by finding the cycle length $T$ that minimizes the ratio of the replenishment cost $G(\cdot)$ at the expected state at time $T$ to the cycle length $T$. That is,
\begin{eqnarray}
\alpha_{G} = \min_{T \geq 0} \frac{G(\mathbb{E}[\bm{I}(T-\tau)])}{T}~,\label{eq:alpha_G}
\end{eqnarray}
where $\bm{I}(t)$ denotes the random state at time $t$ and $\mathbb{E}[\bm{I}(t)]=(Q_1-\lambda_1 t , \dots, Q_J-\lambda_J t)$  is the expected vector of inventory levels at time $t$ for Poisson demands with rate $\lambda_j t$. 

Our second approximation builds on Theorem~\ref{thm:cycle_decomp}. While the optimal approach would require the calculation of entrance fees for all states in the continue set, our approximation $\alpha_{\hat{g}}$ evaluates these fees only at the expected states after each transition of the uniformized chain. We then use the number of transitions $N$ that minimizes the ratio of this expression to the cycle length, that is, 
\begin{eqnarray}
\alpha_{\hat{g}} = \min_{N} \frac{\sum_{n=0}^{N} \hat{g}(\mathbb{E}[I(n/\Lambda)]) +\Lambda \; G({{\bm i^0}})}{N+\tau\;\Lambda}~,\label{eq:alpha_ghat}
\end{eqnarray}
where $n/\Lambda$ is the expected time for $n$ transitions of the uniformized chain. 

Both approximations are used in an online implementation following the structure of the optimal policy. That is, policies $\mu_{\alpha_G}$ and $\mu_{\alpha_{\hat{g}}}$ are defined as
\begin{eqnarray}
\mu_{\alpha_G}({\bm i}) = \begin{cases}
\text{continue} & \text{if } \hat{g}({\bm i}) \leq \alpha_{G}~, \\
\text{stop} & \text{if } \hat{g}({\bm i}) > \alpha_{G}~,
\end{cases}\quad \text{ and } \quad 
\mu_{\alpha_{\hat{g}}}({\bm i}) = \begin{cases}
\text{continue} & \text{if } \hat{g}({\bm i}) \leq \alpha_{\hat{g}}~, \\
\text{stop} & \text{if } \hat{g}({\bm i}) > \alpha_{\hat{g}}~.
\end{cases}
\end{eqnarray}

The true long-term average costs of these policies, $c(\mu_{\alpha_G})$ and $c(\mu_{\alpha_{\hat{g}}})$, can be evaluated using standard Markov chain analysis. For our example in Table \ref{table:state_info}, both approximations yield policies that achieve the optimal performance. The policy based on $\alpha_G$ uses policy parameter $\alpha_G = 6.9199$, while the policy based on $\alpha_{\hat{g}}$ uses policy parameter $\alpha_{\hat{g}} = 7.2830$. Both policies result in the same optimal long-term average cost $c(\mu_{\alpha_G})= c(\mu_{\alpha_{\hat{g}}})= 7.3832 = \alpha^*$.

\begin{restatable}{proposition}{policybounds}
\label{thm:policy_bounds}

Let $\mu_{\alpha^*}$ be the exact online control policy with long-term average cost $c(\mu_{\alpha^*}) = \alpha^*$, and let $\mu_{\alpha_{\hat{g}}}$ and $\mu_{\alpha_G}$ be policies from the approximate online control framework based on  on the approximated policy parameters $\alpha_{\hat{g}}$ and $\alpha_G$, respectively. The long-term average costs of these policies satisfy the following. 
\begin{enumerate}
    \item By the optimality of $\mu_{\alpha^*}$, 
       $ c(\mu_{\alpha^*}) \leq c(\mu_{\alpha_{\hat{g}}}) \quad \text{and} \quad c(\mu_{\alpha^*}) \leq c(\mu_{\alpha_G})~. $
    \item For any approximated policy parameters $\alpha'$ (either using $\alpha_{\hat{g}}$ or $\alpha_G$), 
    \begin{eqnarray}
        &\text{if } \alpha' \geq \alpha^*, &\text{ then } \alpha' \geq c(\mu_{\alpha'}) \geq \alpha^* = c(\mu_{\alpha^*}) ~, \text{ and } \label{eq:upper_bound}\\
        &\text{if } \alpha' \leq \alpha^*, &\text{ then } \alpha' \leq \alpha^* = c(\mu_{\alpha^*}) \leq c(\mu_{\alpha'}) ~. \label{eq:lower_bound}
    \end{eqnarray}
\end{enumerate}
\end{restatable}

These relationships demonstrate that while the approximated policy parameters may over- or underestimate $\alpha^*$, the resulting policy costs adjust in the opposite direction. The proof of Proposition~\ref{thm:policy_bounds} is provided in the Electronic Companion~\ref{sec:Proofs for Trigger Set Control Policy}. As this result predicts, we show that these approximations perform almost optimally using a set of numerical tests.

\section{Optimal Fixed Cycle Replenishment Policy}
\label{sec:periodic_policy}

In this section, we formulate the problem of determining the length of an optimal {\it fixed cycle} to replenish IVMs. Under this policy, the objective of long-run expected average cost per unit time for a fixed cycle length of $T$ can be written as
\begin{eqnarray}
C_F(T, \bm Q) = \frac{A +\sum_{j=1}^{J} b_j \mathbb{E}[S_j(T, Q_j)]}{T}~,
\end{eqnarray}
where $\mathbb{E}[S_j(T, Q_j)]$ is the expected shortage of product $j$ during $T$ units of time, for starting item~$j$ inventory of $Q_j$ units. Before presenting the analysis, we introduce two reasonable assumptions.

\begin{assumption}
  We consider only strictly positive slot allocations, that is, $Q_j > 0$ for all items~$j$.
\end{assumption}

\begin{assumption}\label{assum:leadtime_omit}
In deriving the optimal cycle time $T^*$, we omit the lead time $\tau$ from our analysis, since it is constant and does not affect the optimization of the objective function.
\end{assumption}
Assumption \ref{assum:leadtime_omit} allows us to consider any positive continuous value of $T$. However, if the resulting optimal cycle time $T^*$ is less than $\tau$, it would violate Assumption~\ref{assum:no_overlap} from Section~\ref{sec:optimal_policy}, which prevents overlapping replenishment orders. We address this consideration when establishing bounds on the fixed cost $A$ in Section~\ref{sec:numerical}.

Before characterizing the optimal cycle length, we first establish the limiting behavior of the cost function as the cycle length grows arbitrarily large. This helps us understand the worst-case performance of the system and provides context for the optimization problem.

\begin{restatable}{lemma}{thmasymptotic}\label{thm:asymptotic}
The long-run average cost approaches a finite limit as fixed the cycle length increases, 
\begin{eqnarray}
\lim_{T \to \infty} C_{F}(T,\bm Q)&=& \sum_{j=1}^{J} b_j\;\lambda_j~.\nonumber
\end{eqnarray}
\end{restatable}

This asymptotic result shows that even with arbitrarily long cycles, the cost remains bounded, converging to the sum of the products of demand rates and stockout costs. Having established this upper bound on costs, we now turn to the question of finding an optimal cycle length that minimizes the long-run average cost. The following theorem establishes both the existence and uniqueness of such an optimal cycle length under a reasonable condition on the fixed cost.

\begin{restatable}{theorem}{thmoptimalcycle}\label{thm:optimal_cycle}
For a fixed replenishment cost of $A$ and item-specific shortage costs $b_j$ for all $j$, if $$A<\sum_{j=1}^{J}\;b_j\;Q_j~,$$ there exists a unique optimal cycle length $T^*$ that minimizes the long-run average cost. This optimal cycle length is characterized by the equation
\begin{equation}
\sum_{j=1}^{J} b_j Q_j P(D_j(T^*) \geq Q_j+1) = A.
\end{equation}
\end{restatable}

The proofs of Lemma \ref{thm:asymptotic}
and Theorem  \ref{thm:optimal_cycle}, along with supporting technical results are provided in the Electronic Companion~\ref{sec: Proofs for Fixed Cycle Replenishment}.

\section{Numerical Study}
\label{sec:numerical}

We now present an extensive numerical study to compare the performance of the trigger set control policy with that of fixed cycle replenishment.
We first note that the optimal trigger set control policy naturally outperforms the best fixed cycle replenishment due to its state-dependent nature. That is, when we ensure that both policies adhere to the same lead time $\tau$ (meaning that the fixed cycle length cannot be smaller than $\tau$ to prevent overlapping replenishments, as established in Assumption~\ref{assum:no_overlap}), the optimal state-dependent approach achieves lower or equal costs. The reason is that, while fixed cycle replenishment follows a predetermined schedule regardless of inventory status, the trigger set control policy makes replenishment decisions based on real-time inventory levels, allowing for more responsive and efficient restocking.

In our numerical experiments, since the trigger set control policy $c(\mu_{\alpha})$ may allow inventory to decrease indefinitely, we introduce a bounded version $c(\mu_{\alpha}, I_{\text{min}})$ that forces replenishment when the inventory of any item reaches $I_{\text{min}}$. This boundary not only ensures the ergodicity of the Markov chain but also creates a finite state space that is computationally tractable. For our numerical study, we set $I_{\text{min}} = -2$, allowing at most two units of backorders while ensuring ergodicity and computational feasibility.

For each type of policy, we test instances with two to six items, running 100 trials per instance. The parameters for each trial are randomly generated within the ranges chosen to reflect realistic vending machine operations while keeping the computation feasible: initial inventory of ${Q_j \in \{1, 2, \ldots, 20\}}$ units, stockout costs ${b_j \in \{3, 4, \ldots, 10\}}$ dollars per unit, demand rates ${\lambda_j \in \{1, 2, 3, 4\}}$ units per day. We set the lead time at $\tau = 1$ day, reflecting the fact that replenishments in this context are typically made from a local warehouse.

To ensure a fair comparison between fixed cycle and bounded trigger set control policies, we must address two key differences in their structures.
First, since the trigger set control policy assumes a lead time $\tau$ while the fixed cycle replenishment does not, we need to prevent the fixed cycle replenishment from replenishing more frequently than possible under the trigger set control policy. This is accomplished by setting a minimum fixed cost~$A_{\min}$ that naturally ensures that the optimal cycle length will exceed~$\tau$ without explicitly constraining~$T$. The same~$A_{\min}$ also ensures that it is optimal for the trigger set control policy to not replenish in the initial state, aligning the behavior of both policies. The proofs establishing these bounds are provided in the Electronic Companion~\ref{TC; numerical}.
\begin{restatable}{proposition}{propinitialcont}\label{prop:initial_cont}
Continuing without replenishment is optimal in the initial state if and only if
\begin{equation}
A > A_{\min} = \sum_j b_j \; Q_j \; P(D_j(\tau) \geq Q_j +1).
\end{equation}
\end{restatable}
By setting the fixed cost $A > A_{\min}$, we ensure that Assumptions~\ref{assum:initial_state} and \ref{assum:no_overlap} are always satisfied. These assumptions state that the cycle time would be at least  $\tau$ for the fixed cycle replenishment and a replenishment is not triggered when the IVM is full under an optimal trigger policy.

Second, while the trigger set control policy is bounded by $I_{\text{min}}$, the fixed cycle replenishment does not have such a bound. To ensure that optimal fixed cycle replenishment is in compliance with these bounds without having to explicitly bound it, we define~$T_{\text{max}}$ as the time until the fastest depleting product reaches $I_{\text{min}}$ considering expected demand. That is,
\begin{equation} 
T_{\text{max}} = \min\limits_{j=1,\ldots,J} \frac{Q_j - I_{\text{min}}}{\lambda_j}.
\end{equation}

\begin{restatable}{proposition}{propmaxfixed}\label{prop:max_fixed}
For a given maximum cycle time $T_{\text{max}}$, the fixed cost~$A$ should satisfy
\begin{equation}
A < A_{\max} = \sum\limits_{j=1}^J b_j \; Q_j \; P(D_j(T_{\text{max}})\geq Q_j +1),
\end{equation}
to ensure the optimal cycle length $T^*$ does not exceed $T_{\text{max}}$ established for fixed cycle replenishment.
\end{restatable}

For each experiment, we randomly generated the fixed cost~$A$ between $A_{\min}$~and~$A_{\max}$. This ensures that the optimal policies of both types operate within comparable ranges while respecting their structural differences.

Computations were performed using the Swan cluster at the Holland Computing Center (HCC), the high-performance computing core for the University of Nebraska. Each compute node is equipped with dual Intel Xeon Gold 6348 processors (2.60 GHz) with 56 cores and 256GB RAM.

To compute the bounded optimal trigger set control policies and their corresponding costs, we compared two exact solution methods: solving the linear program (LP) formulation defined in Equation~(\ref{eq:LP_Formulation}) using the CPLEX solver and Algorithm \ref{alg:continue_set} based on the ordering of states by $\hat{g}(\bm i)$. Our analysis (given in Appendix~\ref{app:comp_performance}) shows that state ordering is significantly faster while achieving identical results; therefore, we employ this method for our subsequent analysis.

\begin{table}[htp]
\TABLE
{Relative Suboptimality (\%) from $c(\mu_{\alpha^*},I_{\text{min}})$ by Number of Items\label{tab:rel_gaps}}
{\begin{tabular}{@{}l@{\quad}l@{\quad}c@{\quad\quad}c@{\quad\quad}c@{\quad\quad}c@{\quad\quad}@{}}
\hline\up
Items & Method & Mean & Std Dev & Min & Max \\
\hline\up
2 & $c(\mu_{\alpha_{\hat{g}}}, I_{\text{min}})$ & 0.01 & 0.05 & 0.00 & 0.47 \\
  & $c(\mu_{\alpha_G}, I_{\text{min}})$ & 0.04 & 0.16 & 0.00 & 1.49 \\
  & $C_F(T^*,\bm{Q})$ & 9.79 & 6.81 & 0.00 & 33.82 \\
\hline\up
3 & $c(\mu_{\alpha_{\hat{g}}}, I_{\text{min}})$ & 0.01 & 0.02 & 0.00 & 0.10 \\
  & $c(\mu_{\alpha_G}, I_{\text{min}})$ & 0.04 & 0.09 & 0.00 & 0.54 \\
  & $C_F(T^*,\bm{Q})$ & 7.59 & 5.71 & 0.02 & 30.73 \\
\hline\up
4 & $c(\mu_{\alpha_{\hat{g}}}, I_{\text{min}})$ & 0.01 & 0.02 & 0.00 & 0.14 \\
  & $c(\mu_{\alpha_G}, I_{\text{min}})$ & 0.06 & 0.11 & 0.00 & 0.55 \\
  & $C_F(T^*,\bm{Q})$ & 5.53 & 4.62 & 0.04 & 22.76 \\
\hline\up
5 & $c(\mu_{\alpha_{\hat{g}}}, I_{\text{min}})$ & 0.02 & 0.12 & 0.00 & 1.18 \\
  & $c(\mu_{\alpha_G}, I_{\text{min}})$ & 0.06 & 0.07 & 0.00 & 0.33 \\
  & $C_F(T^*,\bm{Q})$ & 4.27 & 4.02 & 0.00 & 17.81 \\
\hline\up
6 & $c(\mu_{\alpha_{\hat{g}}}, I_{\text{min}})$ & 0.02 & 0.05 & 0.00 & 0.31 \\
  & $c(\mu_{\alpha_G}, I_{\text{min}})$ & 0.07 & 0.09 & 0.00 & 0.42 \\
  & $C_F(T^*,\bm{Q})$ & 3.55 & 3.12 & 0.04 & 20.11 \down\\
\hline
\end{tabular}}
{}
\end{table}

Table~\ref{tab:rel_gaps} reports the relative suboptimality, comparing the policy costs relative to the cost of the bounded optimal trigger set control policy $c(\mu_{\alpha^*},I_{\text{min}})$. Suboptimality is calculated as $100\%~\times~\frac{\text{value}~-~c(\mu_{\alpha^*},I_{\text{min}})}{ c(\mu_{\alpha^*},I_{\text{min}})}$, with positive values indicating percentage cost increase relative to the optimal policy. We examine the cost of optimal fixed cycle replenishments $C_F(T^*,\bm{Q})$ as well as the realized bounded policy costs of the trigger set control policy approximations $c(\mu_{\alpha_{\hat{g}}}, I_{\text{min}})$ and $c(\mu_{\alpha_{G}}, I_{\text{min}})$.
Although the approximated policy parameters~$\alpha_{\hat{g}}$ and $\alpha_{G}$ typically underestimate and overestimate the exact parameter $\alpha^*$, their realized policy costs $c(\mu_{\alpha_{\hat{g}}}, I_{\text{min}})$ and $c(\mu_{\alpha_{G}}, I_{\text{min}})$ remain remarkably close to the optimal cost $c(\mu_{\alpha^*},I_{\text{min}})=\alpha^*$.

Table~\ref{tab:comp_times} presents the average computation times across our 100 trials for each solution approach and number of items.
The approximate online control framework using either $\alpha_{\hat{g}}$ or $\alpha_{G}$ and the fixed cycle replenishment calculation $C_F(T^*,\bm{Q})$ offer significant computational advantages for determining policy parameters, especially for larger instances, and produce easily implementable policies. As shown in Table~\ref{tab:comp_times}, these methods are remarkably efficient, with computation times increasing modestly as the number of items grows.

In stark contrast, the exact bounded optimal trigger set control policy $c(\mu_{\alpha^*},I_{\text{min}})$ exhibits exponential growth in computation time, scaling significantly with increasing problem number of items. For a problem with six items, the exact method requires nearly 75 seconds - approximately 1,375 times longer than calculating $\alpha_{\hat{g}}$ and 186,000 times longer than calculating $\alpha_{G}$ for the same problem size. The reason for this exponential growth in computation time is that the state space size equals $\prod_{j=1}^n (Q_j - I_{\text{min}} + 1)$, which grows exponentially with the number of items considered.
Even for moderate problem sizes with 7-10 items, computing the exact optimal policy becomes not just time-consuming but computationally infeasible, making our approximation methods essential for problems with hundreds of different items encountered in practice.

It is important to note that while our approximation methods efficiently calculate the policy parameters $\alpha_{\hat{g}}$ and $\alpha_{G}$, computing the actual long-term average costs $c(\mu_{\alpha_{\hat{g}}}, I_{\text{min}})$ and $c(\mu_{\alpha_{G}}, I_{\text{min}})$ when using these parameters still requires solving the underlying Markov chain—a process computationally similar to finding the bounded optimal cost $c(\mu_{\alpha^*},I_{\text{min}})$. 

\begin{table}[htp]
\TABLE
{Average Computation Times (seconds) by Number of Items\label{tab:comp_times}}
{\begin{tabular}{@{}l@{\quad}ccccc@{}}
\hline\up
Method & 2 & 3 & 4 & 5 & 6 \\
\hline\up
$\alpha_{\hat{g}}$ & 0.0186 & 0.0273 & 0.0358 & 0.0452 & 0.0542 \\
$\alpha_{G}$       & 0.0003 & 0.0003 & 0.0003 & 0.0004 & 0.0004 \\
$C_F(T^*,\bm{Q})$  & 0.0023 & 0.0034 & 0.0045 & 0.0056 & 0.0067 \\
$c(\mu_{\alpha^*},I_{\text{min}})$ 
                  & 0.0035 & 0.0388 & 0.3385 & 5.6927 & 74.5324 \down\\
\hline
\end{tabular}}
{}
\end{table}

Our results highlight the fundamental trade-offs between policy optimality and computational complexity. The following observations summarize our findings.

\begin{observation}
   Fixed cycle replenishment enables immediate cost evaluation without complex Markov chain calculations, making it valuable when exact cost analysis is required.
\end{observation}

\begin{observation}
   Fixed cycle replenishment provides significant practical benefits in real-world applications, as fixed replenishment schedules are easier to implement and coordinate than dynamic replenishment alerts. These operational advantages can lead to cost savings that are not captured in our theoretical analysis.
\end{observation}

\section{Case Study using Actual Transactional Data}
\label{sec:casestudy}

In this section, we evaluate the performance of our optimal/near-optimal policies using 
actual transaction data from 2023, covering 30 IVMs at a single customer location.  To derive the problem parameters for our analysis, we started with reference data that specified which items were stocked in each vending machine and their respective slot capacities (i.e., $\bm Q$). An item-device combination represents a unique pairing of an item with a specific IVM - for example, safety gloves stored in Machine A would be one item-device combination, while the same safety gloves stored in Machine B would be another combination. We then matched these reference data with the transaction records from 2023. 
During this process, we identified and removed four transactions in which the dispensed quantity exceeded the specified slot capacity, assuming that these were either data entry errors or reflected changes in slot allocations that were not captured in our reference data.

From this cleaned dataset, we obtained the problem parameters for each device-product combination by setting the stockout cost $\bm b$ to the most recent unit price observed in the transaction data. The daily demand rate for each item, $\bm \lambda$ was calculated by summing the total quantity dispensed over the year and dividing by 365 days, while the slot capacities $Q_j$ were taken directly from the reference data. This process resulted in 316 item-device combinations with complete parameter sets (i.e., $\bm b$, $\bm \lambda$, $\bm Q$), forming the basis for our policy evaluations.

To evaluate these policies, we implemented two different approaches. For fixed cycle replenishments, we divided the year into fixed time periods of length $T$ and chronologically processed the transaction data within each period. At the start of each period, we replenished all items to their capacity, then used the actual transaction data during that period to track inventory depletion and calculate stockout costs whenever demand exceeded the available inventory. Similarly, for trigger set control policies, we processed transactions chronologically, maintaining a continuous record of inventory levels and triggering replenishments whenever the state function $\hat{g}(\bm{i})$ exceeded its respective policy parameter. In both cases, stockout costs were incurred when demand exceeded available inventory. For the trigger set control policy, additional stockouts could occur during the lead time~$\tau$ between when a replenishment was triggered and when it was completed.

Using these implementations with a lead time $\tau$ set to 0.5 days, we tested three types of policies for different values of fixed cost,~$A$: fixed cycle replenishments with predetermined periods (i.e., $T=2$,~$4$,~and~$7$ days), a policy using the theoretically optimal cycle length $T^*$ calculated from Theorem~\ref{thm:optimal_cycle}, and trigger set control policies resulting from the approximate online control framework using approximated policy parameters $\alpha_G$ and $\alpha_{\hat{g}}$ calculated from Equations~\eqref{eq:alpha_G} and~\eqref{eq:alpha_ghat}.

\begin{figure}[htbp]
\FIGURE
{\includegraphics[width=0.8\linewidth]{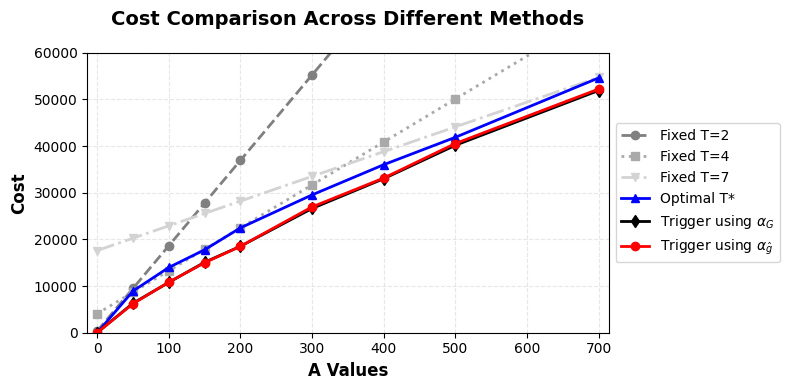}}
{Cost comparison of policies with different fixed costs ($A$) for 
industrial case study\label{fig:cost_comp_total}}
{Using 
actual transaction data, the plot shows total annual costs as a function of fixed replenishment cost~$A$. Fixed cycle replenishments with $T=2,3,4$ days demonstrate how shorter cycles become increasingly expensive as fixed costs rise, while the optimal $T^*$ policy and the trigger set control policies using policy parameters $\alpha_G$ and $\alpha_{\hat{g}}$ achieve lower costs.}
\end{figure}

Our industrial partner currently visits most of its customer sites daily to restock IVMs. Figure~\ref{fig:cost_comp_total} compares the empirical annual costs that result from applying our policies to actual transaction data. The results demonstrate that even at relatively low fixed costs, high-frequency policies (e.g., daily or every two days) quickly become expensive due to the accumulation of fixed costs. In contrast, both the optimal $T^*$ policy and the trigger set control policies demonstrate lower costs across all fixed cost values. Tables~\ref{tab:comprehensive_comparison} and~\ref{tab:policy_comparison} provide detailed breakdowns of policy performance under different fixed cost values.

For fixed cycle replenishment, Table~\ref{tab:policy_comparison} shows detailed cost breakdowns across different fixed costs $A$, including the number of replenishment cycles, fixed costs, and stockout costs. Table~\ref{tab:comprehensive_comparison} further analyzes the performance of the policy in both continuous and integer constrained cycle lengths. Our analysis reveals several key patterns that we summarize in the following. 
\begin{observation}
The empirically found minimum cycle length and the theoretically optimal cycle length $T^*$ produce similar empirical total costs, with the difference typically being less than 10\%. Furthermore, when cycle lengths are limited to integer values for operational feasibility, empirically optimal integer cycles frequently align with theoretical predictions.
\end{observation}

\begin{observation}
The theoretical model consistently predicts lower costs than those observed in practice. This discrepancy likely stems from the actual demand patterns deviating from our Poisson assumption. In particular, we observe that when an item is demanded on a given day, it tends to be demanded multiple times that day - for instance, when safety gloves are needed, multiple workers typically require them. 
\end{observation}
Despite these deviations from theoretical predictions, our analysis reveals that the theoretical model provides valuable practical guidance.

\begin{table}
\TABLE
{Fixed Cycle Replenishment Performance (Annual)\label{tab:comprehensive_comparison}}
{\begin{tabular}{@{}r@{\quad}|@{\quad}cccc@{\quad}|@{\quad}cccc@{}}
\hline\up
 & \multicolumn{4}{c@{\quad}|@{\quad}}{Continuous} & \multicolumn{4}{c}{Integer-Constrained} \\
\hline\up
A & $T^*$ & Cost & $T^e$ & Cost & $T^*_{int}$ & Cost & $T^e_{int}$ & Cost \\
\hline\up
100 & 4.52 & 13,990.37 & 3.82 & 12,544.41 & 4 & 13,286.15 & 4 & 13,286.15 \\
150 & 4.79 & 17,779.59 & 3.49 & 17,469.27 & 5 & 18,852.89 & 4 & 17,886.15 \\
200 & 4.99 & 22,477.17 & 4.89 & 21,329.03 & 5 & 22,502.89 & 4 & 22,486.15 \\
400 & 5.50 & 36,012.84 & 6.10 & 35,582.86 & 5 & 37,102.89 & 6 & 36,243.43 \\
500 & 5.68 & 41,881.00 & 6.08 & 41,579.31 & 6 & 42,343.43 & 6 & 42,343.43 \\
700 & 5.97 & 54,610.44 & 6.57 & 53,155.51 & 6 & 54,543.43 & 6 & 54,543.43 \down\\
\hline
\end{tabular}}
{The table compares empirical costs when using different cycle lengths. For continuous and integer cycle lengths, we show costs using both the theoretical optimal ($T^*$) and empirically best ($T^e$) cycles.}
\end{table}

For the trigger set control policy, we used our two approximations, $\alpha_G$ and $\alpha_{\hat{g}}$. As shown in Table~\ref{tab:policy_comparison}, both variants achieve comparable empirical total costs. 

\begin{observation}
Despite slight differences between the values of $\alpha_G$ and $\alpha_{\hat{g}}$, both variants of the approximate online control policy
achieve comparable empirical total costs. This shows that the online policy is relatively robust. 
Moreover, both variants consistently outperform the fixed cycle replenishment across all tested fixed cost values.
\end{observation}

\begin{table}
\TABLE
{Comparison of fixed cycle and trigger set control policies for different fixed costs $A$.\label{tab:policy_comparison}}
{\begin{tabular}{@{}r@{\quad}|@{\quad}l@{\quad}|@{\quad}rrrr@{\quad}|@{\quad}r@{}}
\hline\up
A & Policy & Cycles & Fixed & Stockout & Total & Theoretical \\
\hline\up
100 & Fixed Cycle ($T=1$) & 365 & 36,500.00 & 25.98 & 36,525.98 & 36,503.41 \\
    & Fixed Cycle ($T =T^*$) & 81 & 8,100.00 & 5,890.37 & 13,990.37 & 9,597.33 \\
    & Trigger $\alpha_G$ & 91 & 9,100.00 & 1,686.27 & 10,786.27 & 7,467.90 \\
    & Trigger $\alpha_{\hat{g}}$ & 89 & 8,900.00 & 1,946.62 & 10,846.62 & 7,862.10 \\
\hline\up
150 & Fixed Cycle ($T=1$) & 365 & 54,750.00 & 25.98 & 54,775.98 & 54,753.41 \\
    & Fixed Cycle ($T = T^*$) & 77 & 11,550.00 & 6,229.59 & 17,779.59 & 13,508.66 \\
    & Trigger $\alpha_G$ & 85 & 12,750.00 & 2,354.58 & 15,104.58 & 10,829.55 \\
    & Trigger $\alpha_{\hat{g}}$ & 85 & 12,750.00 & 2,250.99 & 15,000.99 & 11,362.45 \\
\hline\up
200 & Fixed Cycle ($T=1$) & 365 & 73,000.00 & 25.98 & 73,025.98 & 73,003.41 \\
    & Fixed Cycle ($T =T^*$) & 74 & 14,800.00 & 7,677.17 & 22,477.17 & 17,236.87 \\
    & Trigger $\alpha_G$ & 79 & 15,800.00 & 2,726.11 & 18,526.11 & 14,089.00 \\
    & Trigger $\alpha_{\hat{g}}$ & 78 & 15,600.00 & 2,925.48 & 18,525.48 & 14,735.05 \\
\hline\up
400 & Fixed Cycle ($T=1$) & 365 & 146,000.00 & 25.98 & 146,025.98 & 146,003.41 \\
    & Fixed Cycle ($T =T^*$) & 67 & 26,800.00 & 9,212.84 & 36,012.84 & 31,108.43 \\
    & Trigger $\alpha_G$ & 64 & 25,600.00 & 7,414.60 & 33,014.60 & 26,502.65 \\
    & Trigger $\alpha_{\hat{g}}$ & 64 & 25,600.00 & 7,550.26 & 33,150.26 & 27,455.30 \\
\hline\up
500 & Fixed Cycle ($T=1$) & 365 & 182,500.00 & 25.98 & 182,525.98 & 182,503.41 \\
    & Fixed Cycle ($T =T^*$) & 65 & 32,500.00 & 9,381.00 & 41,881.00 & 37,640.69 \\
    & Trigger $\alpha_G$ & 61 & 30,500.00 & 9,672.20 & 40,172.20 & 32,536.10 \\
    & Trigger $\alpha_{\hat{g}}$ & 60 & 30,000.00 & 10,550.29 & 40,550.29 & 33,572.70 \\
\hline\up
700 & Fixed Cycle ($T=1$) & 365 & 255,500.00 & 25.98 & 255,525.98 & 255,503.41 \\
    & Fixed Cycle ($T =T^*$) & 62 & 43,400.00 & 11,210.44 & 54,610.44 & 50,172.99 \\
    & Trigger $\alpha_G$ & 50 & 35,000.00 & 16,913.81 & 51,913.81 & 44,311.00 \\
    & Trigger $\alpha_{\hat{g}}$ & 49 & 34,300.00 & 17,920.72 & 52,220.72 & 45,544.70 \down\\
\hline
\end{tabular}}
{The table shows empirical cost components and number of replenishment cycles for each policy type, along with theoretical cost predictions/approximations $C_F(T=1, \bm Q),C_F(T^*, \bm Q),\alpha_G, \alpha_{\hat{g}}$. Fixed cycle replenishments are shown for both daily replenishment ($T=1$, current practice) and optimal cycle length ($T^*$), while trigger set results use policy parameters $\alpha_G$ and $\alpha_{\hat{g}}$. All costs are reported on an annual basis.}
\end{table}

Based on the data in Table~\ref{tab:policy_comparison}, the fixed cycle replenishment using optimal cycle lengths ($T^*$) dramatically outperforms the current practice of daily replenishment, showing potential cost reductions of 61.7 to 78.6\% across different fixed cost values. The trigger set control policies provide additional improvements of 4.1 to 22.9 \% compared to the optimal fixed cycle approach. While these results suggest substantial potential savings, it is important to note that our industrial partner's current daily replenishment practice is driven by their strong emphasis on service quality and minimizing stockouts. Nevertheless, from a purely cost perspective, our approach demonstrates significant potential for cost reduction, though any implementation would need to carefully balance these savings against service quality considerations.

The case study demonstrates that our methodologies can effectively handle industrial problems in the real world involving hundreds of items across multiple IVMs. The trigger set control policy, using either approximation, consistently achieves lower costs than fixed cycle replenishment, showing the potential for significant cost savings in practice. Although fixed cycle replenishments may offer operational advantages that could lead to additional savings not captured in our analysis, our theoretical results provide valuable guidance for both policy types. For fixed cycle replenishments, the theoretical optimal cycle length closely matches empirically optimal frequencies, while for trigger set control policies, our approximations prove both computationally efficient and highly effective at reducing total costs. 

\section{Conclusions and Future Work}
\label{sec:Conclusions}

This paper addresses replenishment challenges in contexts that utilize industrial vending machines, a growing technology in the industrial distribution sector. We develop and analyze two replenishment policies: a trigger set control policy and a fixed cycle replenishment. Our theoretical analysis of these control policies provides several key insights. First, we prove that the optimal trigger set control policy exhibits a monotonicity property, allowing efficient implementation through an approximate online control framework. Second, we derive the optimal cycle length for fixed cycle replenishment and show that it can effectively balance fixed replenishment costs against increasing stockout risks.

Testing these policies using actual transaction data reveals important practical insights. Although both policies show gaps between theoretical predictions and actual performance, these gaps are relatively small, typically less than 10\% for fixed cycle replenishment. Moreover, when constrained to integer cycle lengths for operational feasibility, empirically optimal cycles frequently align with theoretical predictions, particularly at higher fixed costs. These findings suggest that our theoretical framework can provide valuable practical guidance.

The trigger set control policy, which leverages the real-time stock level monitoring capabilities of IVMs, consistently outperforms the fixed cycle replenishment in our tests, with both approximations ($\alpha_G$ and $\alpha_{\hat{g}}$) achieving comparable results. However, it is worthwhile to note that fixed cycle replenishments may offer practical advantages in implementation (e.g., simpler scheduling and coordination), and these operational benefits may make fixed cycle replenishments more attractive.  

Our work opens several promising directions for future research. First, extending our analysis to incorporate capacity constraints in delivery vehicles would add practical value, as this is a common constraint in real-world operations. Second, our analysis of actual transaction data reveals that demand often occurs in batches - when an item is demanded on a given day, it is typically demanded multiple times. This observation suggests that a compound Poisson demand process might better represent the actual demand patterns than our current Poisson assumption. Third, the development of methods to optimize slot allocations represents a significant opportunity for cost reduction. The slot allocation problem involves determining how many slots to assign to each product within an IVM. Our LP formulation to optimize the trigger set provides a promising foundation for this extension. By incorporating slot allocations as decision variables in the LP formulation, we could simultaneously solve optimal replenishment policies and slot allocations. By jointly optimizing slot allocations and replenishment policies, firms could achieve substantially lower operating costs while maintaining high service levels. 

Our results provide industrial distributors with theoretically grounded yet practical tools for optimizing IVM replenishment operations. The framework we develop can help industrial distribution firms balance operational costs against service levels while considering constraints such as practically meaningful (integer) cycle lengths and lead times.

\ACKNOWLEDGMENT{We gratefully acknowledge our industrial partner for providing the data that made this research possible. Their support and collaboration have been invaluable in advancing our study. The views and conclusions expressed in this study are solely those of the authors and do not necessarily reflect the opinions of any company.
}


\bibliographystyle{informs2014} 
\bibliography{biblioKarina} 





%
%
%
\newpage

\ECSwitch 

\ECHead{E-Companion: Proofs and Additional Details}

\section{Notation}
\label{sec:notation}
\begingroup
\renewcommand{\arraystretch}{1} 
\small
\begin{longtable}{|p{0.09\textwidth}|p{0.87\textwidth}|}
\caption{Notation and definitions} \label{tab:notation} \\
\hline
\multicolumn{2}{|l|}{\textbf{System Parameters}} \\
\hline
\endfirsthead
\multicolumn{2}{l}{\textit{Continued from previous page}} \\
\hline
\endhead
\hline
\multicolumn{2}{r}{\textit{Continued on next page}} \\
\endfoot
\hline
\endlastfoot
$J$ & Number of unique items $J$\\
$Q_j$ & Number of slots allocated to item $j$ \\
$\bm{Q}$ & Vector of slot allocations $(Q_1, Q_2, \ldots, Q_J)$ \\
$\lambda_j$ & Demand rate for item $j$ per unit time \\
$\Lambda$ & Total demand rate, $\sum_{j=1}^J \lambda_j$ \\
$p_j$ & Probability of demand for item $j$, equal to $\lambda_j/\Lambda$ \\
$\tau$ & Lead time for replenishment \\
\hline
\multicolumn{2}{|l|}{\textbf{Cost Parameters}} \\
\hline
$A$ & Fixed cost per replenishment \\
$b_j$ & Stockout cost per unit of item $j$ \\
$A_{\min}$ & Minimum ordering cost making it optimal to continue in initial state \\
$A_{\max}$ & Maximum ordering cost ensuring finite  replenishment cycles\\
\hline
\multicolumn{2}{|l|}{\textbf{State Space and Policy Variables}} \\
\hline
$\bm {i}$ & System state vector $(i_1, i_2, \ldots, i_J)$ where $i_j$ is inventory level of item $j$ \\
$\bm{i^{-j}}$ & System state vector $i$  with one fewer unit of item $j$ $(i_1, i_2, \ldots, i_j -1, \ldots, i_J)$ \\
$\bm{i^0}$ & Initial state (typically equals $\bm{Q}$) \\
$\bm{i^n}$ & State after $n$ transitions \\
$\mathcal{S}$ & State space \\
$\mathcal{U}$ & Set of states where replenishment is triggered (trigger set) \\
$\overline{\mathcal{U}}$ & Set of continuation states (no replenishment) \\
$\mathcal{U}_r$ & Set of required stopping states \\
$\mathcal{T}_r$ & Set of required stopping states in entrance fee problem \\
$\mathcal{T}_c$ & Set of required continuation states in entrance fee problem \\
$I_{\text{min}}$ & Minimum allowed inventory level \\
$\bm I(t)$ & System state at time $t$ \\
$k^*$ &  Critical state index under optimal trigger set control policy (the index that determines the boundary between continuation and triggering replenishment) \\
\hline
\multicolumn{2}{|l|}{\textbf{Random Variables and Functions}} \\
\hline
$D_j(t)$ & Random demand for item $j$ during time interval $t$ \\
$G(\bm {i})$ & Expected immediate cost of stopping in state $\bm {i}$ \\
$\hat{g}(\bm {i})$ & Incentive fee function that determines optimal stopping decisions \\
$f'(\bm {i})$ & Entrance fee at state $\bm {i}$ \\
$f(\bm {i})$ & Cost of continuing one more step from state $\bm {i}$ \\
$F(\bm {i})$ & Terminal payoff function in stopping problem \\
$F'(\bm {i})$ & Terminal payoff function in entrance fee problem (which is equal to zero) \\
$\kappa(\bm {i})$ & Index mapping function $\kappa: \mathcal{S} \rightarrow \{1,\ldots,|\mathcal{S}|\}$ that orders states by increasing $\hat{g}(\bm {i})$ values \\
$\kappa^{-1}(k)$ & State vector with $k$th smallest entrance fee, inverse of function $\kappa(\bm {i})$  \\
$P(\bm m|\bm i)$ & Transition probability from state $\bm i$ to state $\bm m$ \\
$P_{\mathcal{U}}(\bm m|\bm i)$ & Transition probability under stopping set $\mathcal{U}$ \\
$\delta(\bm m,\bm{Q})$ & Binary indicator function (1 if $k = \bm{Q}$, 0 otherwise) \\
$h(k)$ & Non-decreasing function used in monotonicity proofs \\
$\rho_{\bm i}$&  Probability of being in state $\bm i$ during a cycle\\
\hline
\multicolumn{2}{|l|}{\textbf{Performance Measures and Policies}} \\
\hline
$\alpha^*$ & Optimal long-run expected average cost per unit time \\
$\alpha_G$ & Approximated $\alpha^*$ based on expected final stopping cost \\
$\alpha_{\hat{g}}$ & Approximated $\alpha^*$ based on entrance fees \\
$C_F(T,\bm{Q})$ & Long-run expected average cost per unit time for fixed cycle replenishment with cycle of $T$ time units \\
$T^*$ & Optimal cycle length for fixed cycle replenishment \\
$T^*_{\text{Int}}$ & Integer-valued optimal cycle length \\
$T_{\text{max}}$ & Minimum time until fastest-depleting product reaches $I_{\text{min}}$ \\
$C(\bm{i}^k)$ & Cost incurred at state $\bm{i}$ visited at step $k$ \\
$T(\bm{i}^k)$ & Time spent in state $\bm{i}$ visited at step $k$ \\
$TC$ & Random cost per cycle \\
$T$ & Random length of cycle \\
$S_j(T,Q_j)$ & Expected shortage of item $j$ during time $T$ with initial inventory $Q_j$ \\
$\mu_{\alpha}(\bm {i})$ & Online control policy using $\alpha$ as the decision parameter\\
$c(\mu_{\alpha})$ & Long-term average cost under policy $\mu_{\alpha}$ \\
$c(\mu, I_{\text{min}})$ & Long-term average cost of policy $\mu$ bounded by minimum allowed inventory level $I_{\text{min}}$\\
\hline
\multicolumn{2}{|l|}{\textbf{Variables in Linear Programs}} \\
\hline
$\pi_{\bm i}$ & Stationary probability distribution for state $\bm i$ \\
$\pi_{\bm i}(m)$ & Proportion of periods in state $\bm i$ during first $m$ transitions \\
$r_{\bm i}$ &  Linearization variable for state $\bm i$ \\
$y_{\bm i}$ & Linearization variable $r_{\bm i}$ for stopping states ($\bm i \in \mathcal{U}$) \\
$z_{\bm i}$ & Linearization variable $r_{\bm i}$ for continuation states ($\bm i \in \overline{\mathcal{U}}$) \\
\end{longtable}
\endgroup

\section{Proofs for the Trigger Set Control Policies}
\label{sec:Proofs for Trigger Set Control Policy}
This section provides proofs for the results in Section \ref{sec:optimal_policy}.


First, we derive the long-run expected average cost per unit time given a trigger set $\mathcal{U}$:

\avgcostprop*
\begin{proof}
{\it Proof. }
Using uniformization rate of ~$\Lambda$ to model the state transitions and denoting the state visited after $k$ transitions as ${\bm i^k}$, the long-run expected average cost per unit time for trigger set~$\cal U$ is 
\begin{eqnarray}
    \lim_{m \to \infty} \frac{\frac{1}{m}\sum_{k=0}^{m-1}C({\bm i^k})}{\frac{1}{m}\sum_{k=0}^{m-1}T({\bm i^k})}~,
\end{eqnarray}
where $C({\bm i^k})$ and $T({\bm i^k})$ denote the cost incurred and the amount of time spent in state ${\bm i^k}$, respectively. The terms $\frac{1}{m}\sum_{k=0}^{m-1}C({\bm i^k})$ and $\frac{1}{m}\sum_{k=0}^{m-1}T({\bm i^k})$ represent sample averages which, by the ergodic theorem for Markov chains, converge to their respective expected values as $m \to \infty$. This formulation allows us to express the long-run average cost using the limiting behavior of these sample means. We can write this expression as 
\begin{eqnarray}
    \lim_{m \to \infty} \frac{\frac{1}{m}\sum_{k=0}^{m-1}C({\bm i^k})}{\frac{1}{m}\sum_{k=0}^{m-1}T({\bm i^k})} &=& \lim_{m \to \infty} \frac{\sum_{\bm i} C(\bm i)\pi_{\bm i}(m)}{\sum_{\bm i }T(\bm i)\pi_{\bm i}(m)}~,
\end{eqnarray}
where $\pi_{\bm i}(m)$ denotes the proportion of times we are in state $\bm i $ during the first $m$ transitions. Since the chain defined by this process is ergodic with infinite visits to $\bm{Q}$ (i.e., full IVM), we have $\lim_{m \to \infty} \pi_{\bm i}(m) = \pi_{\bm i}$, defining the stationary distribution at transitions of the discrete-time Markov chain. We are ready to derive the expression for the expected long-term average cost per unit time under the trigger set $\cal U$ as
\begin{eqnarray}
    \alpha_\cal U=\lim_{m \to \infty} \frac{\sum_{\bm i} C(\bm i)\pi_{\bm i}(m)}{\sum_{\bm i }T(\bm i)\pi_{\bm i}(m)} = \frac{\sum_{\bm i} C(\bm i)\pi_{\bm i}}{\sum_{\bm i }T(\bm i)\pi_{\bm i}} 
    &=&\frac{\sum_{\bm i \in \cal {\bar U}} C(\bm i)\pi_{\bm i}+\sum_{\bm i \in \cal {U}}C(\bm i)\pi_{\bm i}}{\sum_{\bm i \in \cal {\bar U}} T(\bm i)\pi_{\bm i}+\sum_{\bm i \in \cal {U}}T(\bm i)\pi_{\bm i}} \nonumber \\
    &=&\frac{\sum_{\bm i \in \cal {U}}G(\bm i)\pi_{\bm i}}{({1}/{\Lambda})\;\sum_{\bm i \in \cal {\bar U}} \pi_{\bm i}+\tau\;\sum_{\bm i \in \cal {U}}\pi_{\bm i}} ~\label{longterm_avg_cost_alpha},
\end{eqnarray}
which follows from the fact that average time at ``continue'' states is $1/\Lambda$, time at all triggering states is $\tau$, and costs equal to $G(\bm i)$ are only incurred at triggering states. The stationary distribution, $\pi_{\bm i}$ for all $\bm i \in \cal S$ can be obtained by solving 
\begin{eqnarray}
    \pi_{\bm m} &=& \sum_{\bm i} {P}_{\cal U}(\bm m|{\bm i}) \pi_{\bm i}, \quad \text{ and } \quad  
    \sum_{\bm i} \pi_{\bm i} =1 ~,
\end{eqnarray}
The one-step transition probability from state $\bm{i}$ to $\bm{m}$ under the stopping set $\cal U$, i.e., ${P}_{\cal U}({\bm m}|{\bm i})$, is 
\begin{eqnarray}
    {P}_{\cal U}({\bm m}|{\bm i})=\begin{cases}
        {P}({\bm m}|{\bm i})~, \quad &\text{ if } {\bm i} \notin \cal U~, \\
        \delta({\bm m},\bm Q)~, \quad &\text{ if } {\bm i} \in \cal U~, \\
    \end{cases}
\end{eqnarray}
where the binary indicator function, $\delta(\bm m,\bm Q)=1$ if $\bm m= \bm Q$ and takes a value of zero otherwise. The transition probabilities of the uniformized chain, denoted as $P(\bm m|{\bm i})$, are relatively simple; transitions from continue states $\bm i$ are only possible to states with one fewer inventory of one of the $J$ items. Arrival of demand for item $j$ transitions the system to state ${\bm i}^{-j}=(i_i, i_2, \dots, i_j-1, \dots, i_J)$, which occurs with probability $p_j = \lambda_j/\Lambda$. 

Plugging the definition of $P_{\cal U}(\bm m|\bm i)$ above, we get the equations for the stationary distribution under trigger set ${\cal U}$ in terms of the original transition probabilities of the uniformized chain ${P}(\bm m|\bm i)$~as 
\begin{eqnarray}
\displaystyle
\pi_{\bm m} = \delta(\bm m,\bm Q) \sum_{\bm i \in \cal U} \pi_{\bm i} + \sum_{\bm i \in \overline{\cal U}} {P}(\bm m|\bm i) \pi_{\bm i} ~, \quad \text{ and } \quad 
\sum_{\bm i \in \cal U} \pi_{\bm i} + \sum_{\bm i \in \overline{\cal U}} \pi_{\bm i} =1~,
\end{eqnarray}
which gives the given expression for the long-run average cost per unit time.~\qed
\end{proof}

We now provide details on the formulation of the linear program. 
\linearization*

\begin{proof}
{\it Proof. }
Using $\pi_{\bm i} =\frac{r_{\bm i}}{\left ( \frac{1}{\Lambda}-(\tau-1) \sum_{\bm m \in \cal U} r_{\bm m} \right )} $,  we can rewrite the stationary distribution equations as
\begin{eqnarray}
    &r_{\bm m} = \delta(\bm m,\bm Q) \sum_{{\bm i} \in \cal U} r_{\bm i} + \sum_{{\bm i} \in \overline{\cal U}} {P}(\bm m|\bm i) r_{\bm i}~,  \quad \text{ and } \quad 
    \frac{1}{\Lambda} \sum_{{\bm i} \in \overline{\cal U}} r_{\bm i} +\tau \sum_{{\bm i} \in \cal U} r_{\bm i} =\frac{1}{\Lambda}~.
\end{eqnarray}
The closed-form expression for the average cost under the trigger set $\cal U$ given in Equation~(\ref{eq:objfn}) is then equal to $\Lambda\;\sum_{\bm i\in {\cal U}} G(\bm i) r_{\bm i}$, as shown below. 
\begin{eqnarray}
\frac{\sum_{{\bm i}\in {\cal U}} G({\bm i}) \pi_{\bm i}}{ \frac{1}{\Lambda}\sum_{{\bm i} \in {\overline{\cal U}}} \pi_{\bm i}+ \tau \sum_{{\bm i} \in {\cal U}} \pi_{\bm i}}
&=&\frac{\sum_{{\bm i}\in {\cal U}} G({\bm i})\frac{r_{\bm i}}{\left ( \frac{1}{\Lambda}-(\tau-1) \sum_{\bm m \in \cal U} r_{\bm m} \right )}}{ \frac{1}{\Lambda}\sum_{{\bm i} \in {\overline{\cal U}}} \frac{r_{\bm i}}{\left ( \frac{1}{\Lambda}-(\tau-1) \sum_{\bm m \in \cal U} r_{\bm m} \right )}+ \tau \sum_{{\bm i} \in {\cal U}} \frac{r_{\bm i}}{\left ( \frac{1}{\Lambda}-(\tau-1) \sum_{\bm m \in \cal U} r_{\bm m} \right )}}~, \nonumber \\
&=&\frac{\sum_{{\bm i}\in {\cal U}} G({\bm i})r_{\bm i}}{ \frac{1}{\Lambda}\sum_{{\bm i} \in {\overline{\cal U}}}{r_{\bm i}}+ \tau \sum_{{\bm i} \in {\cal U}} {r_{\bm i}}} = \Lambda\;\sum_{{\bm i}\in {\cal U}} G({\bm i})\;r_{\bm i}~, 
\end{eqnarray}
as stated in the lemma.~\qed
\end{proof}

Next, we work toward deriving a closed-form expression for the entrance fee $f'(\bm i)$ from Lemma~\ref{lemma:entrance_fee_expression}. To achieve this, we first establish a supporting lemma (Lemma~\ref{lemma: G}) that characterizes how the cost function $G(\cdot)$ changes with inventory reductions.

\begin{lemma}\label{lemma: G}
For any state $\bm i$ referring to a state vector of $(i_1, \dots, i_j,\dots, i_J)$ and where $i^{-j}$ refers to the state $(i_1, \dots, i_j-1,\dots, i_J)$, we have, for any items $l,m \in \{1,2,\dots,j, \dots, J\}$, 
\begin{eqnarray}
G(\bm i^{-l}) -   G((\bm i^{-l})^{-m}) = - b_m P(D_m (\tau) \geq i_m)  =  G(\bm i) -   G(\bm i^{-m})~.
\end{eqnarray}
\end{lemma}

\begin{proof}
{\it Proof.}
We analyze two cases and show that they lead to the same expression.

\begin{enumerate}
    \item First, consider $G(\bm i) - G(\bm i^{-m})$. Using the definition of $G(\cdot)$:
    \begin{eqnarray}
    G(\bm i) &=& A + \sum_{j \in J} b_j \sum_{r=i_j}^{\infty} (r-i_j) P(D_j(\tau)= r) ~, \\
    G(\bm i^{-m}) &=& A + \sum_{j \in J: j \neq m} b_j \sum_{r=i_j}^{\infty} (r-i_j) P(D_j(\tau)= r) \\
    && \qquad \qquad \qquad + b_m \sum_{r=i_m-1}^{\infty} (r-i_m-1) P(D_m(\tau) = r).\nonumber
    \end{eqnarray}
    
    Subtracting we get, 
    \begin{eqnarray}
    G(\bm i) - G(\bm i^{-m}) &=& b_m \left[\sum_{r=i_m}^{\infty} (r-i_m) P(D_j(\tau)= r) -\sum_{r=i_m-1}^{\infty} (r-i_m-1) P(D_m(\tau) = r) \right]~,\nonumber\\
                    &=& -b_m P(D_m (\tau) \geq i_m).
    \end{eqnarray}

\item Next, consider $G(\bm i^{-j}) - G((\bm i^{-j})^{-m})$. Similarly,
    \begin{eqnarray}
    G(\bm i^{-l}) &=& A + \sum_{j \in J: j \neq l} b_j \sum_{r=i_j}^{\infty} (r-i_j) P(D_j(\tau)= r)\\
    && \qquad \qquad \qquad + b_l \sum_{r=i_{l}-1}^{\infty} (r-i_{l}-1) P(D_l(\tau)= r)~, \nonumber \\
    G((\bm i^{-l})^{-m}) &=& A + \sum_{j \in J: j \neq l, m} b_j \sum_{r=i_j}^{\infty} (r-i_j) P(D_j(\tau)= r)~, \\
    &&  \qquad + b_l \sum_{r=i_{l}-1}^{\infty} (r-i_{l}-1) P(D_l(\tau) = r) + b_m \sum_{r=i_m-1}^{\infty} (r-i_m-1) P(D_m(\tau) = r).\nonumber
    \end{eqnarray}
Subtracting we get,
    \begin{eqnarray}
    G(\bm i^{-l}) - G((\bm i^{-l})^{-m}) &=& b_m \left[\sum_{r=i_m}^{\infty} (r-i_m) P(D_j(\tau)= r) -\sum_{r=i_m-1}^{\infty} (r-i_m-1) P(D_m(\tau) = r) \right]~, \nonumber \\
    &=& -b_m P(D_m (\tau) \geq i_m)~.
    \end{eqnarray}
\end{enumerate}
Thus, we have shown that both expressions equal $-b_m P(D_m (\tau) \geq i_m)$, completing the proof.~\qed
\end{proof}

Using this property, we can derive a simpler expression for the entrance fee as follows. 
\entrancefeeexpression*
\begin{proof}
{\it Proof. }
From the transformation to an entrance fee problem we have
\begin{eqnarray}
f'(\bm i) &=& f(\bm i) - \left[\sum_j p_j F({\bm i^{-j}}) - F(\bm i)\right]~,\nonumber\\
&=& - \frac{\alpha^*}{\Lambda} - \left[\sum_j p_j (-G({\bm i^{-j}})) - (-G(\bm i))\right]~,\nonumber\\
&=& - \frac{\alpha^*}{\Lambda} + \left[G(\bm i) - \sum_j p_j G({\bm i^{-j}})\right]~.
\end{eqnarray}

By Lemma~\ref{lemma: G}, we know that
\begin{eqnarray}
G(\bm i) - \sum_j p_j G({\bm i^{-j}}) = \sum_j p_j b_j P(D_j(\tau) \geq i_j)~.
\end{eqnarray}
Substituting this result,
\begin{eqnarray}
f'(\bm i) &=& - \alpha^* \; \frac{1}{\Lambda} + \sum_j p_j b_j P(D_j(\tau) \geq i_j)~.
\end{eqnarray}
This completes our derivation of the simpler expression for the entrance fee.~\qed
\end{proof}

Next, we work toward proving Theorem~\ref{thm:threshold}, which establishes that an optimal replenishment policy follows a monotonicity structure based on state indices. To demonstrate this, we first establish a crucial monotonicity property of the entrance fee function $f'(\cdot)$ in the following lemma.

\begin{restatable}{lemma}{lemmafprime}\label{lemma: f'}
For any state $\bm i$ and any item $j$, $f'(\bm i^{-j}) \geq f'(\bm i)$.
\end{restatable}

\begin{proof}
{\it Proof.}
We expand $f'(\bm i^{-j})$ and use the result from Lemma~\ref{lemma: G}, as follows.
\begin{eqnarray}
   f'(\bm i^{-j}) 
   &=&  (-g(\bm i^{-j}) - \alpha^*) - \left [G(\bm i^{-j}) - \sum_k p_k G((\bm i^{-j})^{-k}) \right ]~,\nonumber\\
   &= &  - \alpha^* - \sum_k p_k [G(\bm i^{-j}) - G((\bm i^{-j})^{-k})]~,  \nonumber\\
   &=& - \alpha^* + \sum_k p_k b_k\; P(D_k\geq i_k^{-j}), \nonumber \\
   & \geq &-  \alpha^*+ \sum_k p_k b_k\; P(D_k\geq i_k)  = f'(\bm i). 
\end{eqnarray}
The inequality follows from the fact that $P(D_k\geq i_k^{-j}) \geq P(D_k\geq i_k)$ since $i_k^{-j} \leq i_k$ for all $k$.~\qed
\end{proof}

These lemmas provide the foundation for proving Theorem~\ref{thm:threshold}, which establishes the existence of a critical state index that determines the optimal replenishment policy. We now proceed with the proof of this main result.

\threshold*
\begin{proof}
{\it Proof. }
Proof by demonstrating that the entrance fee problem satisfies the conditions of Theorem 10.3 in \citep{breiman1964stopping}. The conditions required for the theorem to hold are
\begin{enumerate}
\item [i.] The entrance fee $f'(\kappa^{-1}(k))$ is non-decreasing in state index, $k$.
\item [ii.] For any non-decreasing function $h(k)$, $\sum_l h(l)P(\kappa^{-1}(l)|\kappa^{-1}(k))$ is also non-decreasing in state index~$k$.
\end{enumerate}

Condition (i) is satisfied by our problem formulation, as we index the states such that $f'(\kappa^{-1}(k))$ is non-decreasing in state index~$k$. It is worth noting that while $f'(\cdot)$ includes the term $\alpha^*$ (the minimum long-run expected average cost per unit time under the optimal policy), which is unknown a priori, we can replace it with an arbitrary constant without affecting the ordering of the states.


We know that the index of  state $\kappa({\bm i^{-j}})$ is at least equal to $\kappa(\bm i)+1$ (i.e., one more than the index of the state $\bm i$) since $f'({\bm i^{-j}}) \geq f'(\bm i)$ by Lemma~\ref{lemma: f'}. We can now show that  $\sum_l h(l) P(\kappa^{-1}(l)|\kappa^{-1}(k))$ is nondecreasing in index $k$. 

Let $\bm i$ denote the state vector that corresponds to the index $k$, i.e., $\kappa^{-1}(k)=\bm i$. Recall that function $h$ is non-decreasing in state index.  
\begin{eqnarray}
   \sum_l h(l) P(\kappa^{-1}(l)|\kappa^{-1}(k))
   &=&  \sum_j p_j h(\kappa([\kappa^{-1}(k)]^{-j})) \quad \text{(due to the structure of the transitions)}~, \nonumber\\
   &\leq& \sum_j p_j h(\kappa([\kappa^{-1}(k+1)]^{-j}))~, \label {eq:buyuk} \nonumber\\
   &=& \sum_l h(l) P(\kappa^{-1}(l)|\kappa^{-1}(k+1))~.
\end{eqnarray}
The inequality in~(\ref{eq:buyuk}) is due to the fact that (i) $f'(\bm i) = f'(\kappa^{-1}(k))\leq f'(\kappa^{-1}(k+1) \leq f'(\bm i^{-j})$ and (ii) the function $h$ is nondecreasing in index $k$. Proving Conditions (i) and (ii) implies that there exists a monotonic optimal solution with a stopping set that includes all states with indices $k \geq k^*$.~\qed
\end{proof}


Next, we will prove Corollary \ref{cor:optimal_continue_set}, which characterizes the optimal continue and stopping sets more precisely in terms of the entrance fee function.

\optimalContSet*

\begin{proof}{\it Proof. }
From Theorem \ref{thm:threshold}, we know that there exists a state index $k^*$ such that it is optimal to continue in all states with indices less than $k^*$ and to trigger replenishment in all states with indices greater than or equal to $k^*$.

Since states are indexed such that $f'(\cdot)$ is non-decreasing in the state index, this critical state index $k^*$ corresponds to a specific value of $f'(\cdot)$. We can show that this critical value corresponds precisely to $f'(\bm i) = 0$.

Assume by contradiction that the critical boundary occurs at some state $\bm i$ with $f'(\bm i) < 0$, but it's optimal to trigger replenishment. In this case, according to the entrance fee interpretation, the expected improvement in stopping reward ($\sum_j p_j F({\bm i^{-j}}) - F(\bm i)$) exceeds the cost of continuing one step ($f(\bm i)$), making it more beneficial to continue than to stop, which contradicts the optimality of triggering replenishment.

Similarly, if the critical boundary occurs at some state $\bm i$ with $f'(\bm i) > 0$, but it is optimal to continue, the cost of continuing exceeds the expected improvement in stopping reward. Furthermore, by the monotonicity property established in Lemma \ref{lemma: f'}, once $f'(\bm i) > 0$ at some state, it remains positive for all subsequent states reachable through demand arrivals. This makes it more beneficial to stop, which contradicts the optimality of continuing.

Therefore, the critical boundary must occur precisely at states where $f'(\bm i) = 0$, with all states satisfying $f'(\bm i) \leq 0$ belonging to the optimal continue set $\overline{\mathcal{U}}^*$ and all states satisfying $f'(\bm i) > 0$ belonging to the optimal stopping set $\mathcal{U}^*$.~\qed
\end{proof}


From Corollary~\ref{cor:optimal_continue_set}, we know that the optimal continue set consists exactly of those states where $f'(\bm i) \leq 0$. 
Having established this characterization of the optimal policy, we can now show how it enables an efficient online implementation through the following lemma.
\onlinepolicy*

\begin{proof} 
{\it Proof. }
We show that $f'({\bm i}) \leq 0 \iff \hat{g}({\bm i}) \leq \alpha^*$ through the following chain of inequalities. 

\begin{eqnarray}
f'({\bm i}) \leq 0 &\iff&
 - \alpha^* \frac{1}{\Lambda} + \sum_j p_j b_j P(D_j(\tau) \geq i_j) \leq 0~, \nonumber\\ 
&\iff& \sum_j p_j b_j P(D_j(\tau) \geq i_j) \leq \alpha^* \frac{1}{\Lambda}~, \nonumber\\ 
&\iff& \sum_j \lambda_j b_j P(D_j(\tau) \geq i_j) \leq \alpha^*~, \nonumber\\ 
&\iff& \hat{g}(\bm i) \leq \alpha^*~.
\end{eqnarray}
This completes the proof of the equivalence between the optimal policy characterization and its online implementation.~\qed
\end{proof}

Next, we will prove Lemma \ref{lemma:incremental_calc}.
\incrementalcalc*
\begin{proof}
{\it Proof. }
 We can calculate $\rho_{\bm i}$ by solving 
\begin{eqnarray}
    \rho_{\bm m} &=& \sum_{\bm i} {P}_{\cal U}(\bm m|{\bm i}) \rho_{\bm i}, \quad \forall \bm m\in \cal S,  \quad \text{ and } \quad  
    \rho_{\bm{i^0}} =1 ~,
\end{eqnarray}
where the one-step transition probability from state $\bm{i}$ to $\bm{m}$ under the stopping set $\cal U$ is 
\begin{eqnarray}
    {P}_{\cal U}({\bm m}|{\bm i})=\begin{cases}
        {P}({\bm m}|{\bm i})~, \quad &\text{ if } {\bm i} \in \cal{\overline{U}}~, \\
        \delta({\bm m},\bm Q)~, \quad &\text{ if } {\bm i} \in \cal U~.   \end{cases}
\end{eqnarray}

By Theorem~\ref{thm:threshold}, when states are added in ascending order of $\hat{g}(\cdot)$, all states that could lead to state $\bm i$ (those with higher inventory levels) must have lower $\hat{g}(\cdot)$ values and therefore are already in the continue set. Therefore, we know that (i) all possible paths to state $\bm i$ are in the continue set, (ii) no paths through states with higher $\hat{g}(\cdot)$ values can lead to $\bm i$, and (iii) the visiting probability $\rho_{\bm i}$ depends only on the probability of the specific sequence of demands needed to reach $\bm i$ from ${\bm {i^0}}$, which is given by the multinomial distribution.

Furthermore, adding state $\bm i$ cannot affect the visiting probabilities of previously added states, since they cannot be reached through state $\bm i$.~\qed
\end{proof}

Building on these visiting probabilities, we can decompose the expected cycle length and cost into components that depend only on the continue states as shown in Theorem \ref{thm:cycle_decomp}, proven next. 
\CycleDecompTheoremName*
\begin{proof}
{\it Proof. } Following Theorem 10.6 in  \cite{breiman1964stopping} we can define our problem as an incentive-fee problem with a constant stopping cost of $G({\bm {i^0}})$ and incentive fee of ${g(\bm i) -\left[G(\bm i) - \sum_j p_j G({\bm i^{-j}})\right]}$.
Each cycle begins with the base cost $G({\bm {i^0}})$. When visiting a continue state $\bm i$ (with probability $\rho_{\bm i}$), the system incurs an incentive fee.
Following Lemma \ref{lemma: G} defined in the Electronic Companion \ref{sec:Proofs for Trigger Set Control Policy} we can see that the incentive fee is equal to $\hat{g}(\bm i)/\Lambda$.
Thus, each continue state contributes an expected additional cost of $(\hat{g}(\bm i)/\Lambda) \rho_{\bm i}$ to the cycle cost. Summing over all continue states produces the expected cost per cycle.

Furthermore, consider a cycle starting from the state ${\bm {i^0}}$ until hitting a stopping state. Each cycle has a lead time of $\tau$ after the initiation of a replenishment.
When the system visits a continue state $\bm i$ (which occurs with probability $\rho_{\bm i}$), it spends an additional time $1/\Lambda$ until the next demand arrives. Therefore, each continue state contributes an additional expected time of $\rho_{\bm i}/\Lambda$ to the cycle length. Summing over all continue states yields the expected cycle length.~\qed
\end{proof}

We now establish two supporting lemmas that will be used to analyze the relationship between policy parameters and their corresponding policy costs in Proposition~\ref{thm:policy_bounds}.


\begin{lemma}\label{lemma:ghat_inequality}
If $\hat{g}(\kappa^{-1}(k)) > c(\mu_{\hat{g}(\kappa^{-1}(k-1))})$, then $\hat{g}(\kappa^{-1}(k)) > c(\mu_{\hat{g}(\kappa^{-1}(k))})$~.
\end{lemma}

\begin{proof}
{\it Proof. }
Let $\bm i = \kappa^{-1}(k)$~. By definition of $\mu_{\hat{g}(\kappa^{-1}(k))}$, this policy includes state $\bm i$ in its continue set, while $\mu_{\hat{g}(\kappa^{-1}(k-1))}$ does not.  Let $TC_{k-1}$ and $T_{k-1}$ be the expected cycle cost and cycle time resulting from policy $\mu_{\hat{g}(\kappa^{-1}(k-1))}$~. Then,
\begin{eqnarray}
TC_k = TC_{k-1} + \hat{g}(\bm i)\rho_{\bm i} \quad \text{ and } \quad T_k = T_{k-1} + \rho_{\bm i}~,
\end{eqnarray}

Given $\hat{g}(\bm i) > c(\mu_{\hat{g}(\kappa^{-1}(k-1))}) = \frac{TC_{k-1}}{T_{k-1}}$,
\begin{eqnarray}
\hat{g}(\bm i) > \frac{TC_{k-1}}{T_{k-1}} &\Rightarrow& \hat{g}(\bm i)\;T_{k-1} > TC_{k-1}~, \nonumber\\
&\Rightarrow& \hat{g}(\bm i)(T_{k-1} + \rho_{\bm i}) > TC_{k-1} + \hat{g}(\bm i)\rho_{\bm i}~, \nonumber\\
&\Rightarrow& \hat{g}(\bm i) > \frac{TC_{k-1} + \hat{g}(\bm i)\rho_{\bm i}}{T_{k-1} + \rho_{\bm i}} = c(\mu_{\hat{g}(\kappa^{-1}(k))})~.
\end{eqnarray}

This completes the proof of the inequality relationship.~\qed
\end{proof}

\begin{restatable}{lemma}{ghatbound}\label{lemma:ghat_bound}
For any state $\bm i$, if $\hat{g}(\bm i) > \alpha^*$, then $\hat{g}(\bm i) > c(\mu_{\hat{g}(\bm i)})$.
\end{restatable}
\begin{proof}
{\it Proof. }
For $k^* + 1$, we know $\hat{g}(\kappa^{-1}(k^* + 1)) > \alpha^* = c(\mu_{\hat{g}(\kappa^{-1}(k^*))})$. By Lemma~\ref{lemma:ghat_inequality}, this implies ${\hat{g}(\kappa^{-1}(k^* + 1)) > c(\mu_{\hat{g}(\kappa^{-1}(k^* + 1))})}$.
For any $k > k^* + 1$, since states are ordered by increasing $\hat{g}(\cdot)$,
\begin{eqnarray}
\hat{g}(\kappa^{-1}(k+2)) > \hat{g}(\kappa^{-1}(k^* + 1)) > c(\mu_{\hat{g}(\kappa^{-1}(k^* + 1))})
\end{eqnarray}
and therefore,
\begin{eqnarray}
\hat{g}(\kappa^{-1}(k+2))  > c(\mu_{\hat{g}(\kappa^{-1}(k^* + 2))})~.
\end{eqnarray}

Applying Lemma~\ref{lemma:ghat_inequality} repeatedly for each $k > k^* + 1$ establishes that $\hat{g}(\kappa^{-1}(k)) > c(\mu_{\hat{g}(\kappa^{-1}(k))})$ for all such states.~\qed
\end{proof}


Using these lemmas, we can now prove Proposition~\ref{thm:policy_bounds}, which characterizes the relationship between approximations and policy costs.

\policybounds*

\begin{proof}{\it Proof. }
The first property and Equation~\eqref{eq:lower_bound} follow directly from the definition of optimality: Since $\mu_{\alpha^*}$ is the optimal policy, it achieves the minimum possible long-term average cost. Any policy that deviates from the optimal continue set by excluding optimal continue states (when $\alpha' \leq \alpha^*$) or including suboptimal continue states (when $\alpha' \geq \alpha^*$) must achieve an equal or higher cost.
Equation~\eqref{eq:upper_bound} follows from Lemma~\ref{lemma:ghat_bound}, which establishes that when $\alpha' \geq \alpha^*$, the cost of the corresponding policy $c(\mu_{\alpha'})$ is bounded above by $\alpha'$ and below by $\alpha^*$.~\qed
\end{proof}

\section{Proofs for Optimal Fixed Cycle Replenishment}
\label{sec: Proofs for Fixed Cycle Replenishment}

This section provides detailed proofs for theorems related to fixed cycle replenishment. We first establish several technical results about the expected shortage function and its derivatives. We then prove the asymptotic behavior of the cost function and the existence of an optimal cycle length.

\subsection{Supporting Results}

We begin with a characterization of the expected shortage function and its properties.

\begin{lemma}\label{lemma:expected_shortage}
The expected shortage of product $j$ during cycle length $T$ is given by
\begin{eqnarray}
\mathbb{E}[S_j(T, Q_j)] &=& \lambda_j\;T\;P(D_j(T) \geq Q_j) - Q_j\;P(D_j(T) \geq Q_j+1)~.
\end{eqnarray}
\end{lemma}

\begin{proof}
{\it Proof. }
 To calculate $\mathbb{E}[S_j(T, Q_j)]$, we use the properties of the Poisson distribution. $$r\;P(D_j(T) = r) = \mathbb{E}[D_j(T)] P(D_j(T) = r-1)= \lambda_j\;T\;P(D_j(T) = r-1)~.$$
\begin{eqnarray}
\mathbb{E}[S_j(T, Q_j)] &=& \sum_{r=Q_j}^{\infty} (r-Q_j)\;P(D_j(T) = r)~,\nonumber \\
 &=& \sum_{r=Q_j}^{\infty} r\;P(D_j(T) = r) -\sum_{r=Q_j}^{\infty} Q_j\;P(D_j(T) = r)~, \nonumber\\
 &=& \sum_{r=Q_j}^{\infty}\lambda_j\;T\;P(D_j(T) = r-1) -\sum_{r=Q_j}^{\infty} Q_j\;P(D_j(T) = r) ~, \nonumber\\
 &=& \lambda_j\;T\;P(D_j(T) \geq Q_j-1) - Q_j\;P(D_j(T) \geq Q_j)~, \nonumber \\
 &=& \lambda_j\;T\;P(D_j(T) \geq Q_j) - Q_j\;P(D_j(T) \geq Q_j+1)~.
\end{eqnarray}
This completes the derivation of the expected shortage expression.~\qed
\end{proof}

The following lemmas establish the derivatives of the expected shortage function, which are crucial for characterizing the optimal cycle length.

\begin{lemma}[First Derivative of Shortage]\label{lem:firstDerivShortage}
The first derivative of the expected shortage function $\mathbb{E}[S_j(T, Q_j)]$ is defined by
\begin{eqnarray}
\frac{d \mathbb{E}[S_j(T, Q_j)]}{d T} &=&\lambda_j\;P(D_j(T) \geq Q_j)~.
\end{eqnarray}
\end{lemma}

\begin{proof}
{\it Proof. }
First, we obtain the derivative of the probability mass function for  Poisson distribution.
\begin{eqnarray}
 \frac{d P(D_j(T) = r)}{d T}&=& \frac{d}{d T} \frac{e^{- \lambda_j\;T} (\lambda_j\;T)^r}{r!}~,\nonumber\\
 &=& \frac{\frac{d}{d T}\left((\lambda_j\;T)^r\right) e^{-\lambda_j\;T} + (\lambda_j\;T)^r \frac{d}{dT}\left(e^{-\lambda_j\;T}\right)}{r!}~, \nonumber\\
 &=& \frac{r(\lambda_j\;T)^{r-1} \lambda_j e^{-\lambda_j\;T} + (\lambda_j\;T)^r e^{-\lambda_j\;T} (-\lambda_j)}{r!}~, \nonumber\\
 &=& \lambda_j \left(\frac{e^{-\lambda_j\;T} (\lambda_j\;T)^{r-1}}{(r-1)!}-\frac{e^{-\lambda_j\;T} (\lambda_j\;T)^r}{r!}\right)~, \nonumber\\
 &=& \lambda_j \left(P(D_j(T) = r-1) - P(D_j(T) = r) \right)~.
\end{eqnarray}

Next, we use this result to find the derivative of the cumulative probability distribution
\begin{eqnarray}
 \frac{d P(D_j(T) \geq r)}{d T}&=& \frac{d}{d T} \sum_{x = r}^{\infty} P(D_j(T) = x)~,\nonumber\\
 &=& \lambda_j \sum_{x = r}^{\infty} \left(P(D_j(T) = x-1) - P(D_j(T) = x)\right)~, \nonumber\\
 &=& \lambda_j\;P(D_j(T) = r-1)~. 
\end{eqnarray}

Finally, we apply the chain rule and combine terms to obtain
\begin{eqnarray}
 \frac{d \mathbb{E}[S_j(T, Q_j)]}{d T}
 &=& \frac{d}{d T} (\lambda_j\;T\;P(D_j(T) \geq Q_j) - Q_j\;P(D_j(T) \geq Q_j+1))~,\nonumber\\
 &=& \lambda_j\;P(D_j(T) \geq Q_j) +\lambda_j^2\;T\;P(D_j(T) = Q_j-1) ~, \nonumber\\
 & & \qquad \qquad - Q_j \lambda_j\;P(D_j(T) = Q_j)~, \nonumber\\
 &=& \lambda_j\;P(D_j(T) \geq Q_j) +\lambda_j^2\;T\;P(D_j(T) = Q_j-1) ~, \nonumber\\
 & & \qquad \qquad - \lambda_j \lambda_j\;T\;P(D_j(T) = Q_j-1)~, \nonumber\\
 &=&\lambda_j\;P(D_j(T) \geq Q_j)~.
\end{eqnarray}
This completes the derivation of the first derivative of the expected shortage function.~\qed
\end{proof}

\begin{lemma}\label{lem:secondDerivShortage}
The second derivative of the expected shortage function $\mathbb{E}[S_j(T, Q_j)]$ is defined by
\begin{eqnarray}
 \frac{d^2 \mathbb{E}[S_j(T, Q_j)]}{d T} &=& \lambda_j^2\;P(D_j(T) = Q_j-1)~.
\end{eqnarray}
\end{lemma}

\begin{proof}
{\it Proof. }
Using the result from Lemma~\ref{lem:firstDerivShortage}
\begin{eqnarray}
 \frac{d^2 \mathbb{E}[S_j(T, Q_j)]}{d T}
 &=& \frac{d}{d T} \frac{d \mathbb{E}[S_j(T, Q_j)]}{d T} ~,\nonumber\\
 &=& \frac{d}{d T} \left(\lambda_j\;P(D_j(T) \geq Q_j)\right)~,\nonumber\\
 &=& \lambda_j^2\;P(D_j(T) = Q_j-1)~. 
\end{eqnarray}
This completes the derivation of the second derivative of the expected shortage function.~\qed
\end{proof}

We now establish the derivatives of the cost function used to prove the existence and uniqueness of the optimal cycle length.

\begin{lemma}[First Derivative of Cost Function]\label{lem:firstDerivativeCostFunction}
The first derivative of the cost function $C_{F}(T,\bm Q)$ is defined by
\begin{eqnarray}
\frac{d C_{F}(T,\bm Q)}{d T}&=&\frac{\left(-A + \sum_{j=1}^{J} b_j\;Q_j\;P(D_j(T) \geq Q_j+1) \right)}{T^2}~.
\end{eqnarray}
\end{lemma}

\begin{proof}
{\it Proof. }
Using the first derivative of the expected shortage function from Lemma~\ref{lem:firstDerivShortage},
\begin{eqnarray}
\frac{d C_{F}(T,\bm Q)}{d T} &=& \frac{T \sum_{j=1}^{J} b_j\;\frac{d\mathbb{E}[S_j(T, Q_j)]}{d T} - A - \sum_{j=1}^{J} b_j\;\mathbb{E}[S_j(T, Q_j)]}{T^2}~,\nonumber\\
&=& \frac{-A + \sum_{j=1}^{J} b_j\;\lambda_j\;T\;P(D_j(T) \geq Q_j)}{T^2} 
    \nonumber\\
&& \quad \qquad \qquad - \frac{\sum_{j=1}^{J} b_j\;\left(\lambda_j\;T\;P(D_j(T) \geq Q_j) + Q_j\;P(D_j(T) \geq Q_j+1)\right)}{T^2}~, \nonumber\\
&=&\frac{-A + \sum_{j=1}^{J} b_j\;Q_j\;P(D_j(T) \geq Q_j+1)}{T^2}~.
\end{eqnarray}
This completes the derivation of the first derivative of the cost function.~\qed
\end{proof}

\begin{lemma}[Second Derivative of Cost Function]\label{lem:secondDerivativeCostFunction}
The second derivative of the cost function $C_{F}(T,\bm Q)$ is given by
\begin{eqnarray}
\frac{d^2 C_{F}(T,\bm Q)}{d T^2} &=&\frac{2\;( A - \sum_{j=1}^J b_j\;Q_j\;P(D_j(T) \geq Q_j+1))}{T^3} \nonumber \\
&&\qquad \qquad \qquad + \frac{\sum_{j=1}^J b_j\;\lambda_j^2\;P(D_j(T) = Q_j-1)}{T}~.
\end{eqnarray}
\end{lemma}

\begin{proof}
{\it Proof. }
Using Lemma~\ref{lem:secondDerivShortage} and Lemma~\ref{lem:firstDerivativeCostFunction}, we have
\begin{eqnarray}
\frac{d^2 C_{F}(T,\bm Q)}{d T^2} &=& \frac{d}{d T}\frac{d C_{F}(T,\bm Q)}{d T}\nonumber\\
&=&\frac{2\;(A + \sum_{j=1}^J b_j\;\mathbb{E}[S_j(T, Q_j)])}{T^3} - \frac{2 \sum_{j=1}^J b_j\;\frac{d\mathbb{E}[S_j(T, Q_j)]}{d T}}{T^2} + \frac{\sum_{j=1}^J b_j\;\frac{d^2\mathbb{E}[S_j(T, Q_j)]}{d T^2}}{T}\nonumber\\
&=& \frac{2\;(A + \sum_{j=1}^J b_j\;(\lambda_j\;T\;P(D_j(T) \geq Q_j) - Q_j\;P(D_j(T) \geq Q_j+1))}{T^3} \nonumber \\
&&\qquad \qquad \qquad - \frac{2\;T \sum_{j=1}^J b_j\;(\lambda_j\;P(D_j(T) \geq Q_j))}{T^3} + \frac{\sum_{j=1}^J b_j\;\lambda_j^2\;P(D_j(T) = Q_j-1)}{T}~, \nonumber\\
&=& \frac{2\;( A - \sum_{j=1}^J b_j\;Q_j\;P(D_j(T) \geq Q_j+1))}{T^3} + \frac{\sum_{j=1}^J b_j\;\lambda_j^2\;P(D_j(T) = Q_j-1)}{T}~.
\end{eqnarray}
This completes the derivation of the second derivative of the cost function.~\qed
\end{proof}

\subsection{Other Results on Fixed Cycle Replenishment}

With these technical results, we now prove Lemma \ref{thm:asymptotic} and our main theorem about the fixed cycle replenishment.

\thmasymptotic*
\begin{proof}
{\it Proof. }
Using the expression for expected shortage from Lemma~\ref{lemma:expected_shortage}, we can write
\begin{eqnarray}
C_{F}(T,\bm Q) &=& \frac{A +\sum_{j=1}^{J} b_j \mathbb{E}[S_j(T, Q_j)]}{T}~, \nonumber\\
&=&\frac{A +\sum_{j=1}^{J} b_j\;\left(\lambda_j\;T\; P(D_j(T) \geq Q_j) - Q_j\;P(D_j(T) \geq Q_j+1)\right)}{T}~.
\end{eqnarray}

As $T$ approaches infinity, $P(D_j(T) \geq Q_j)$ approaches 1 for any finite $Q_j$, while the term $\frac{A}{T}$ and the terms involving $Q_j$ approach zero. Therefore,
\begin{eqnarray}
\lim_{T \to \infty} C_{F}(T,\bm Q) &=& \sum_{j=1}^{J} b_j\;\lambda_j~.
\end{eqnarray}
This completes the proof of the asymptotic behavior of the cost function.~\qed
\end{proof}

\thmoptimalcycle*
\begin{proof}
{\it Proof. }
To prove existence and uniqueness, we proceed in three steps: (i) derive the first-order condition using Lemma~\ref{lem:firstDerivativeCostFunction}, (ii) establish monotonicity properties, and (iii) verify the solution is a minimum using Lemma~\ref{lem:secondDerivativeCostFunction}.

Taking the derivative from Lemma~\ref{lem:firstDerivativeCostFunction} and equating to zero, we get

\begin{eqnarray}
\frac{d C_{F}(T,\bm Q)}{d T} = 0&\Rightarrow& \frac{-A + \sum_{j=1}^{J} b_j\;Q_j\;P(D_j(T) \geq Q_j+1)}{T^2} = 0, \nonumber\\
&\Rightarrow& \sum_{j=1}^{J} b_j\;Q_j\;P(D_j(T^*) \geq Q_j+1) = A. 
\end{eqnarray}
To establish uniqueness, we observe two key properties.
First, $ \sum_{j=1}^{J} b_j\;Q_j\;P(D_j(T) \geq Q_j+1) $ is strictly increasing in $T$, which ensures that there can exist at most one value of $T^*$ satisfying the first-order condition.
Second, we examine the limiting behavior and observe that 
\begin{eqnarray}
\lim_{T \to \infty} \sum_{j=1}^{J} b_j\;Q_j\;P(D_j(T) \geq Q_j+1) = \sum_{j=1}^{J} b_j\;Q_j.
\end{eqnarray}

Given these properties when $A < \sum_{j=1}^{J} b_j\;Q_j$, then there exists exactly one value $T^*$ where $\sum_{j=1}^{J} b_j\;Q_j\;P(D_j(T^*) \geq Q_j+1) = A$.

Finally, to verify that this $T^*$ is indeed a minimizer of the cost function, we analyze the second derivative evaluated at $T^*$ using Lemma~\ref{lem:secondDerivativeCostFunction}.
\begin{eqnarray}
\frac{d^2 C_{F}(T,\bm Q)}{d T^2} &=&\frac{2\;( A - \sum_{j=1}^J b_j\;Q_j\;P(D_j(T) \geq Q_j+1))}{T^3} + \frac{\sum_{j=1}^J b_j\;\lambda_j^2\;P(D_j(T) = Q_j-1)}{T}~,\nonumber\\
&=& \frac{\sum_{j=1}^J b_j\;\lambda_j^2\;P(D_j(T) = Q_j-1)}{T}> 0~,
\end{eqnarray}
where the second equality follows by substituting the value of $A$ at the critical point. The positive second derivative confirms that $T^*$ is indeed a minimizer of the cost function.~\qed
\end{proof}

This completes our analysis of the fixed cycle replenishment, establishing both the asymptotic behavior of the cost function and the existence and uniqueness of an optimal cycle length under the specified conditions.

\section{Further Details on the Numerical Study}
\label{TC; numerical}
In our numerical study, we need to ensure fair comparison between the fixed cycle and bounded trigger set control policies by establishing appropriate bounds on the fixed cost $A$. We present two key propositions that determine these bounds.

First, we establish a lower bound on the fixed cost that serves two purposes: (1) it ensures the initial state is part of the continue set, making it optimal to wait rather than immediately replenish in the trigger set control policy, and (2) it guarantees that the optimal cycle length in the fixed cycle replenishment is greater than the lead time $\tau$, preventing replenishment schedules that would be infeasible for the trigger set control policy.

\propinitialcont*
\begin{proof}
{\it Proof. }
Adding initial state ${\bm{i^0}} = {\bm{Q}}$ to the continue set is optimal if
\begin{eqnarray}
    \frac{G({\bm{i^0}})}{\tau} &>& \frac{G({\bm{i^0}})+\hat{g}({\bm{i^0}})\;(1/\Lambda)}{\tau +(1/\Lambda)}
\end{eqnarray}
This inequality can be simplified by cross-multiplication as follows 

\begin{eqnarray}
    \frac{G({\bm{i^0}})}{\tau} &>& \frac{G({\bm{i^0}})+\hat{g}({\bm{i^0}})\;(1/\Lambda)}{\tau +(1/\Lambda)} \Rightarrow G({\bm{i^0}}) > \hat{g}({\bm{i^0}}) \; \tau \nonumber\\
    &\Rightarrow& A + \sum_{j=1}^J b_j \sum_{r=Q_j}^\infty  \left(r-Q_j\right) \; P(D_j(\tau)=r) > \sum_j \lambda_j \; b_j \; P(D_j(\tau) \geq Q_j)\; \tau \nonumber\\
    &\Rightarrow& A > \sum_j b_j \; Q_j \; P(D_j(\tau) \geq Q_j +1).
\end{eqnarray}
This establishes the minimum fixed cost $A_{\min}$ required to make continuing optimal in the initial state.~\qed
\end{proof}

Next, we derive an upper bound on the fixed cost that ensures that the optimal cycle length does not exceed a specified maximum time $T_{\text{max}}$. This is necessary to maintain comparable operating conditions between the two policies.

\propmaxfixed*
\begin{proof}
{\it Proof. }
From Theorem~\ref{thm:optimal_cycle}, we know that the optimal cycle length $T^*$ satisfies
\begin{equation}
\sum_{j=1}^{J} b_j Q_j P(D_j(T^*) \geq Q_j+1) = A.
\end{equation}
Since $\sum_{j=1}^{J} b_j Q_j P(D_j(T) \geq Q_j+1)$ is strictly increasing in $T$ (by Theorem~\ref{thm:optimal_cycle}), this inequality is equivalent to
\begin{eqnarray}
A &\leq& \sum_{j=1}^{J} b_j Q_j P(D_j(T_{\text{max}}) \geq Q_j+1).
\end{eqnarray}
This defines $A_{\max}$ as the upper bound on the fixed cost that ensures $T^* \leq T_{\text{max}}$.~\qed
\end{proof}

Together, these results establish the range of fixed costs $[A_{\min}, A_{\max}]$ within which our numerical comparisons are meaningful and fair.


\subsection{Computational performance comparison}
\label{app:comp_performance}
For the numerical study in Section \ref{sec:numerical}, we compared two exact solution methods for computing optimal trigger set control policies: (i) solving the LP formulation defined in Equation~(\ref{eq:LP_Formulation}) using CPLEX and (ii) a state ordering approach using Algorithm \ref{alg:continue_set}. 
Our results demonstrate that both methods consistently find the same optimal solutions but differ substantially in computational efficiency.

\begin{table}[htbp]
\TABLE
{Computational Performance Comparison\label{tab:comp_performance}}
{\begin{tabular}{|c|cc|cc|cc|cc|}
\hline
& \multicolumn{6}{c|}{Computation Times (seconds)} & \multicolumn{2}{c|}{Scale$^*$} \\
\hline
Items & \multicolumn{2}{c|}{Average} & \multicolumn{2}{c|}{Median} & \multicolumn{2}{c|}{Maximum} & \multicolumn{2}{c|}{} \\
& LP & SO & LP & SO & LP & SO & LP & SO \\
\hline
2 & 0.058 & 0.004 & 0.051 & 0.004 & 0.161 & 0.012 & -- & -- \\
3 & 1.016 & 0.051 & 0.793 & 0.042 & 4.597 & 0.223 & 17.43$\times$ & 12.05$\times$ \\
\hline
\multicolumn{9}{l}{\small $^*$Scale shows increase in mean time from from previous item count} \\
\multicolumn{9}{l}{\small LP: Linear Program (CPLEX), SO: State Ordering} \\
\end{tabular}}
{}
\end{table}

As shown in Table~\ref{tab:comp_performance}, we were able to compare both methods only for problems with 2 and 3 items, as the LP approach became computationally intractable beyond that point even when using the HCC high-performance computing cluster. For these smaller instances, the state-space ordering (SO) method significantly outperforms the linear programming (LP) approach, with performance improvements exceeding 90\%. The performance gap becomes particularly pronounced with 3 items, where SO maintains its efficiency while LP experiences substantial slowdown due to increased model complexity. Importantly, while LP could not handle larger problems, the SO method remains computationally feasible up to problems with 6 items, as demonstrated in our main numerical study. This stark difference in scalability highlights the practical advantage of the SO approach for realistic problem instances.

\end{document}